\documentclass[11pt]{amsart}
\usepackage[margin=1.25in]{geometry}
\usepackage{amsfonts}
\usepackage[alphabetic]{amsrefs}
\usepackage{hyperref}
\usepackage{enumitem}

\theoremstyle{definition}

\makeatletter
\newtheorem*{rep@theorem}{\rep@title}
\newcommand{\newreptheorem}[2]{
\newenvironment{rep#1}[1]{
 \def\rep@title{#2 \ref{##1}}
 \begin{rep@theorem}}
 {\end{rep@theorem}}}
\makeatother

\newtheorem{theorem}{Theorem}[section]
\newreptheorem{theorem}{Theorem}
\newtheorem{lemma}[theorem]{Lemma}

\newtheorem{definition}[theorem]{Definition}
\newtheorem{example}[theorem]{Example}

\newtheorem{remark}[theorem]{Remark}
\numberwithin{equation}{section}

\newtheorem{corollary}[theorem]{Corollary}

\newtheorem*{moquestion}{Motivating Question}
\newtheorem{proposition}[theorem]{Proposition}

\newcommand{\sn}{\mathrm{sn}}
\newcommand{\cs}{\mathrm{cs}}
\newcommand{\II}{I\hspace{-2.5pt}I}

\newcommand{\ep}{\varepsilon}
\renewcommand{\phi}{\varphi}
\renewcommand{\epsilon}{\varepsilon}

\newcommand{\Z}{\mathbb{Z}}
\newcommand{\R}{\mathbb{R}}\newcommand{\C}{\mathbb{C}}\newcommand{\HH}{\mathbb{H}}

\newcommand{\pp}{\mathbb{P}}
\newcommand{\s}{\mathbb{S}}

\DeclareMathOperator{\Hess}{Hess}

\DeclareMathOperator{\vol}{vol}

\newcommand{\of}[1]{\left(#1\right)}

\newcommand{\floor}[1]{\left\lfloor #1 \right\rfloor}

\newcommand{\st}{~|~}

\renewcommand{\hat}{\widehat}
\renewcommand{\tilde}{\widetilde}
\newcommand{\tr}{\mathrm{tr}}

\begin{document}

\title[Weighted Sectional Curvature] 
{The weighted connection and sectional curvature for manifolds with density}

\author{Lee Kennard}
\address{601 Elm Ave., Dept. of Mathematics, University of Oklahoma, Norman, OK 73019}
\email{kennard@ou.edu}
\urladdr{www.math.ou.edu/~kennard}

\author{William Wylie}
\address{215 Carnegie Building\\
Dept. of Math, Syracuse University\\
Syracuse, NY, 13244.}
\email{wwylie@syr.edu}
\urladdr{https://wwylie.expressions.syr.edu}

\author{Dmytro Yeroshkin}
\address{921 S. 8th Ave., Stop 8085\\
Department of Mathematics and Statistics\\
Idaho State University\\
Pocatello, ID 83209}
\email{yerodmyt@isu.edu}
\urladdr{www2.cose.isu.edu/~yerodmyt}

\keywords{}

\begin{abstract}
In this paper we study sectional curvature bounds for Riemannian manifolds with density from the perspective of a weighted torsion free connection introduced recently by the last two authors. We develop two new tools for studying weighted sectional curvature bounds:  a new weighted Rauch comparison theorem  and a modified  notion of convexity for distance functions.  As applications we prove generalizations of theorems of Preissman and Byers for negative curvature,  the (homeomorphic) quarter-pinched sphere theorem, and Cheeger's finiteness theorem.  We also improve results of the first two authors for spaces of positive weighted sectional curvature and symmetry.  
\end{abstract}

\maketitle

\section{Introduction}

Let the triple $(M^n,g, \mu)$ denote an $n$-dimensional Riemannian manifold $(M,g)$ with $\mu$ a smooth measure on $M$.  In \cite{WylieYeroshkin} the last two authors introduced a natural  connection $\nabla^{g, \mu}$ that can be associated to $(M^n,g, \mu)$. It is the unique torsion free connection that both makes $\mu$ parallel and has the same geodesics as the Levi-Civita connection up to reparameterization.  The curvature of the connection gives a $(3,1)$-curvature tensor and a Ricci tensor by the standard formula.  Since  many results in the comparison theory for Riemannian manifolds are statements about geodesics and measure, it is natural to expect that  $\nabla^{g, \mu}$ can be used to give a comparison theory for manifolds with measure.  In \cites{WylieYeroshkin, Wylie} such a comparison theory for the Ricci curvature was investigated.   Despite the fact that lower bounds on the Ricci curvature of $\nabla^{g, \mu}$ are weaker than the Ricci curvature bounds for manifolds with measure that have  previously been considered,  versions of diameter, volume, and Laplacian comparison theorems are recovered.   Rigidity results such as the de Rham and Cheeger-Gromoll splitting theorems and Cheng's maximal diameter theorem are also proven. Some results for Lorentzian metrics have also been established in \cite{WoolgarWylie16}. 

In this paper we are interested in the sectional curvature comparison theory coming from $\nabla^{g, \mu}$.  The choice of the smooth measure $\mu$ is equivalent to choosing a smooth density function. We will normalize the density function $\phi$ such that $\mu = e^{-(n+1) \phi }dvol_g$ where $dvol_g$ is the Riemannian volume element and $n=\dim(M)$.  Then the connection has the  formula, 
\begin{align*}
\nabla^{g, \mu}_XY = \nabla_X Y - d\phi(X) Y - d\phi(Y)X 
\end{align*}
where $\nabla$ is the Levi-Civita connection of $g$.  We will write $\nabla^{g, \phi}$ for $\nabla^{g, \mu}$.  Since we will often think of $g$ as being fixed, we also write $\nabla^{\phi}=\nabla^{g, \phi}$. 

We denote the weighted Riemann curvature tensor by 
\begin{align*}
 R^{\nabla^{\phi}}(X,Y) Z &= \nabla^{\phi}_X \nabla^{\phi}_Y Z - \nabla^{\phi}_Y \nabla^{\phi}_ X Z - \nabla^{\phi}_{[X,Y]} Z,
\end{align*}
whose explicit formula is derived in \cite{WylieYeroshkin}*{Proposition 3.3}. Given two orthonormal vectors $U$ and $V$, we then consider weighted sectional curvature to be the quantity
\begin{align*}
g(R^{\nabla^{\phi}}(V,U)U, V) = \sec(U,V) + \mathrm{Hess} \phi(U,U) + d\phi(U)^2 = \overline{\sec}_{\phi}(U,V). 
\end{align*}
The quantity $\overline{\sec}_{\phi}$ has been studied earlier by the first two authors in \cites{Wylie15, KennardWylie17}. In fact, these works inspired the discovery of  the connection $\nabla^{\phi}$. The main tool used in \cites{Wylie15, KennardWylie17}  is a generalization of the second variation formula. We will see below that using the connection $\nabla^{\phi}$ we can simplify this formula, and use it to establish general Rauch comparison theorems for Jacobi fields.    We also identify a new notion of weighted convexity that is related to bounds on $\overline{\sec}_{\phi}$.   The notion of weighted convexity is somewhat technical  but,  roughly speaking, we show bounds on weighted curvatures give bounds on the Hessian of the distance function in a certain conformal metric, see Section \ref{sec:Convexity} for details.  

We first consider the applications in the cases of positive and negative weighted curvatures. 
  
\begin{definition} Let $(M,g)$ be a Riemannian manifold.  We say that $(M,g)$ has \emph{positive weighted sectional curvature} (PWSC) if there exists a function $\phi$ such that $\overline{\sec}_{\phi}(U,V) > 0$ for all orthonormal pairs of vectors $U,V$.  We say that $(M,g)$ has \emph{negative weighted sectional curvature} (NWSC)  if $\overline{\sec}_{\phi}(U,V) < 0$ for all orthonormal pairs of vectors $U,V$.
\end{definition}

In \cite{KennardWylie17} theorems for manifolds of positive curvature with symmetry, e.g., Weinstein's theorem, the Grove-Searle maximal symmetry rank theorem, and Wilking's connectedness lemma, are established for the weighted curvatures.  We use convexity to  improve the rigidity results in this direction to optimal equivariant diffeomorphism classifications. For example, we obtain the fixed point homogeneous classification of Grove and Searle \cite{GroveSearle97} (see Section \ref{sec:Convexity} for definitions and further remarks):

\begin{theorem}\label{thm:FPH}
Let $(M,g)$ be a simply connected, closed Riemannian manifold with PWSC  that admits a fixed-point homogeneous action by a connected Lie group $G$. Then the $G$--action on $M$ is equivariantly diffeomorphic to a linear action on a compact, rank one symmetric space.
\end{theorem}

We also use our notion of  weighted convexity to establish results for NWSC.   In \cite{Wylie15} it was shown that a space  admitting a function $\phi$ such that $\overline{\sec}_{\phi} \leq 0$ does not have conjugate points, and thus the universal cover must be diffeomorphic to $\mathbb{R}^n$.  We show in this paper that the theorems of  Preissman and Byers for $\pi_1(M)$ are also true for NWSC. 

\begin{theorem}   \label{Thm:Byers} If $(M,g)$ is a compact manifold with NWSC then any solvable subgroup of $\pi_1(M)$ is infinite cyclic and $\pi_1(M)$ does not admit a solvable subgroup of finite index.
\end{theorem} 

We also define non-zero weighted curvature bounds which are, like the notions of PWSC and NWSC, invariants of the metric $(M,g)$.  From the perspective of $\nabla^{\phi}$ the natural form of the curvature bound is to consider curvatures of the form  $R^{\nabla^{\phi}} ( \dot{\gamma}, U, U, \dot{\gamma})$ where $\gamma$ is a geodesic for the connection $\nabla^{\phi}$.  However, since the geodesics of $\nabla^{\phi}$ are not constant speed, this bound translates into a non-constant curvature bound of the form $\overline{\sec}_{\phi} \leq (\geq) k e^{-4\phi}$ where $k$ is a constant. See \cite{WylieYeroshkin}*{Sections 2 \& 3} and Remark \ref{Rem:CurveBound} below for the details. 

We wish to define  $\overline{K}_g$ and $\underline{\kappa}_g$ to be best upper and lower bounds, respectively of  the quantity $e^{4 \phi}\overline{\sec}_{\phi}$  achieved through varying $\phi$ over all smooth real valued functions on $(M,g)$. Rescaling considerations from the $e^{4\phi}$ factor necessitate introducing normalizations  depending on the sign of the bound. See Definition \ref{Def:CurveBound} for  the explicit details.  For the moment we say that there exist explicit  invariants $\overline{K}_g$ and $\underline{\kappa}_g$ of the Riemannian manifold $(M,g)$ such that $\overline{K}_g \leq \sec_{max}(g)$, $\underline{\kappa}_g \geq \sec_{min}(g)$ and such that a compact manifold  $(M,g)$ has PWSC if and only if $\underline{\kappa}_g >0$ and NWSC if and only if $\overline{K}_g < 0$.

For a positive lower bound, we have the following version of Myers' theorem. 

\begin{theorem}  \label{Thm:Myers} Suppose $(M,g)$ is a complete Riemannian manifold with $\underline{\kappa}_g >0$, then $M$ is compact, $\mathrm{diam}(M) \leq \frac{\pi}{\sqrt{\underline{\kappa}_g}}$  and $\pi_1(M)$ is finite. 
\end{theorem}

\begin{remark} As should be expected Theorem \ref{Thm:Myers} is, in fact, true for Ricci curvature, as was proven in \cite{WylieYeroshkin}*{Theorem 2.2}.  Theorem \ref{Thm:Myers} can be seen as a direct corollary of that result, or as a consequence of Lemma \ref{Lem:index} below. 
\end{remark}

Define the weighted pinching constant of a space of positive weighted sectional curvature as $\delta = \underline{\kappa}_g / \overline{K}_g.$  In the next section we will see that,  on a compact manifold,  $\delta \leq 1$. When $\delta > \frac{1}{4}$ we also have the homeomorphic sphere theorem.  

\begin{theorem}\label{Thm:QuarterPinch}
 Let $(M,g)$ be a simply connected  complete manifold  of PWSC and $\delta > \frac{1}{4}$, then $M$ is homeomorphic to the sphere. 
\end{theorem}

We also obtain generalizations of Cheeger's finiteness theorems.  Our proofs require  a pointwise bound on $|d\phi|$.  For $a>0$, we define $\underline{\kappa}_g(a)$ and $\overline{K}_g(a)$ to be the best lower bound and  upper bound respectively among all normalized densities that satisfy $|d\phi|\leq a$.    

Define $\delta (a) = \underline{\kappa}(a) / \overline{K}(a).$  For positive curvature in even dimensions we have the following finiteness result. 
\begin{theorem} \label{IntroThm:FinitePosEven}
For given $n, a>0$ and $0<\delta_0 \leq 1$ the class of Riemannian $2n$-dimensional manifolds with $\underline{\kappa}(a) >0$ and $\delta(a) \geq \delta_0$ contain  only finitely many diffeomorphism types. 
\end{theorem} 

As in the un-weighted setting, in the general case, we also require a lower bound on volume.

\begin{theorem} \label{IntroThm:GenFinite}
For given $n\geq 2$, $a,v, D, k >0$ the class of compact Riemannian manifolds $(M,g)$ with 
\[
\mathrm{diam}(M,g) \leq D, \quad \mathrm{vol}(M,g) \geq v, \quad  \overline{K}_g(a) \leq k, \quad \text{and} \quad  \underline{\kappa}_g(a) \geq -k
\]
contains only finitely many diffeomorphism types. 
\end{theorem}

The paper is organized as follows. In the next section we give the complete definitions of $\underline{\kappa}_g$ and $\overline{K}_g$ and summarize how some earlier results are related to these invariants.  We also discuss some basic examples.  In section 3 we discuss the notion of weighted convexity and apply it to prove Theorems \ref{thm:FPH} and \ref{Thm:Byers} as well as other results about positive and negative weighted curvatures.  In section 4 we prove the Jacobi field comparison theorems for the weighted curvature, including versions of the first and second Rauch theorems (Theorems \ref{Thm:Rauch1} \& \ref{Thm:Rauch2}) as well as a more general weighted version of a Jacobi field comparison due to Heintze-Karcher (Lemma \ref{JacobiVolumeLemma}) which also implies a general comparison for weighted tube volumes (Theorem \ref{Thm:TubeVol}) which may be of independent interest. We also use these comparisons to prove Theorems \ref{Thm:QuarterPinch}, \ref{IntroThm:FinitePosEven} and \ref{IntroThm:GenFinite}.  We finish the paper by also deriving a weighted version of  the Radial Curvature Equation for general hypersurfaces. 

\subsection*{Acknowledgements}
This work was partially supported by NSF Grant DMS-1440140 while the first and second authors were in residence at MSRI in Berkeley, California, during the Spring 2016 semester. The first author was partially supported by NSF Grant DMS-1622541. The second author was supported  by a grant from the Simons Foundation (\#355608, William Wylie) and a grant from the National Science Foundation (DMS-1654034). The third author was partially supported by a grant from the College of Science and Engineering at Idaho State University.
\section{Preliminaries and Examples}

\subsection{Definition of weighted curvature bounds}

In this section we define our weighted generalization of upper and lower curvature bounds.  We consider bounds of the form
\begin{equation} \label{eqn:CurveBound}
\kappa e^{-4\phi} \leq \overline{\sec}_{\phi}\leq K e^{-4\phi} 
\end{equation}
where $\kappa$ and $K$ are constants.  To see that normalization of $\phi$ is needed in (\ref{eqn:CurveBound}) consider adding a constant to $\phi$.  Let  $\psi = \phi + c$ for some constant $c$, then 
\begin{equation}
\label{eqn:AddConstant}
\left(\kappa e^{4c}\right) e^{-4\psi} =  \kappa e^{-4\phi} \leq \overline{\sec}_{\psi} = \overline{\sec}_{\phi} \leq K e^{-4\phi} \leq \left(Ke^{4c}\right) e^{-4\psi}
\end{equation}
This gives us the following proposition. 

\begin{proposition} \label{Prop:NormalizationNeeded}
Let $(M,g)$ be a compact Riemannian manifold, then 
\begin{align*} 
\sup \left\{ \kappa : \exists \phi \text{ s.t. } \overline{ \sec}_{\phi}\geq \kappa e^{-4\phi} \right\}= 0 \text{ or } \infty.
\end{align*}
Moreover, the supremum  $= \infty$ if and only if $(M,g)$ has PWSC. Similarly, 
\begin{align*}
\inf \left\{ K: \exists \phi \text{ s.t. } \overline{ \sec}_{\phi}\leq K e^{-4\phi} \right\} = 0 \text{ or } -\infty. 
\end{align*}
Moreover, the infimum $= -\infty$ if and only if $(M,g)$ has NWSC.
\end{proposition}
 
\begin{proof}
Let $\phi$ be a function such that $\overline{\sec}_{\phi} \geq -\kappa e^{-4\phi}$ for some $\kappa>0$.  Let $\psi_c = \phi + c$. Then, from (\ref{eqn:AddConstant}),  $ \overline{\sec}_{\phi_c} \geq (-\kappa e^{4c}) e^{-4\psi}$.  Letting $c \rightarrow -\infty$ gives us $\sup \left\{ \kappa : \exists \phi \text{ s.t. } \overline{ \sec}_{\phi}\geq \kappa e^{-4\phi} \right\} \geq 0$. 
 
The supremum being greater than zero is equivalent to PWSC by compactness. Then there is a $\kappa >0$ and a $\phi$ such that $\sec_{\phi} \geq \kappa e^{-4\phi}$.  Letting  $\psi_c = \phi + c$ and  $c \rightarrow \infty$ in (\ref{eqn:AddConstant}) gives $\sup \left\{ \kappa : \exists \phi \text{ s.t. } \overline{ \sec}_{\phi}\geq \kappa e^{-4\phi} \right\} = \infty$ in this case. 
 
The second statement about upper bounds is proved in the completely  analogous way. 
\end{proof}
 
Proposition \ref{Prop:NormalizationNeeded} motivates the following definition. 

\begin{definition}   \label{Def:CurveBound} Let $(M,g)$ be a Riemannian manifold.  If $(M,g)$ has PWSC define  
\[ \underline{\kappa}_g = \sup \left\{ \kappa : \exists \phi:M \rightarrow(-\infty, 0] , \overline{ \sec}_{\phi}\geq \kappa e^{-4\phi} \right\},  \]
otherwise, define 
 \[ \underline{\kappa}_g = \sup \left\{ \kappa : \exists \phi:M \rightarrow[0, \infty) , \overline{ \sec}_{\phi}\geq \kappa e^{-4\phi} \right\}.   \]
If $(M,g)$ has NWSC define
\[ \overline{K}_g = \inf \left\{ K: \exists \phi:M \rightarrow (-\infty, 0],\overline{ \sec}_{\phi}\leq K e^{-4\phi} \right\},\]
otherwise define
\[ \overline{K}_g = \inf \left\{ K: \exists \phi:M \rightarrow [0, \infty),\overline{ \sec}_{\phi}\leq K e^{-4\phi} \right\}. \]
\end{definition}

Let   $\sec_{\max}$ and $\sec_{\min}$ be the supremum and infimum of the sectional curvatures of $(M,g)$.  Then by taking $\phi = 0$ we obtain that $\overline{K} \leq \sec_{\max}$ and $\underline{\kappa} \geq \sec_{\min}$.  The bounds on $\phi$ ensure that we can not make the supremums and infimums blow up or shrink to zero simply by adding a constant to the density as in the proof of Proposition \ref{Prop:NormalizationNeeded}.  

The choice of the bounds $\phi \leq 0$ or $\phi \geq 0$ as opposed to some other constant serves to fix a scale for the metric.  For example, if  $(M,g)$ has PWSC and there is a function $\phi$ which is bounded above such that $\overline{\sec}_{\phi} \geq \kappa e^{-4 \phi}$ for some $\kappa>0$, then if we rescale the metric by  $\widehat{g} = e^{-2\phi_{\max}} g$ and modify the density by $\widehat{\phi} = \phi - \phi_{\max}$ then we have $\overline{\sec}_{\widehat g, \widehat{\phi}} = e^{4\phi_{\max}} \overline{\sec}_{\phi, g} \geq \kappa e^{-4\widehat{\phi}}$.  So the rescaled metric will have $\underline{\kappa}_{\widehat{g}} \geq \kappa$. 

Define $\underline{\kappa}_g(a)$ and $\overline{K}_g(a)$ in exactly the same way as $\underline{\kappa}_g$ and $\overline{K}_g$ with the additional assumption that the function $\phi$ must satisfy the derivative bound $|d\phi|_g \leq a$ on $M$.  Then $\underline{\kappa}_g(0)=\sec_{min}$, $\overline{K}_g(0) = \sec_{\max}$,  $\displaystyle \lim_{a \rightarrow \infty} \underline{\kappa}_g(a)=\underline{\kappa}_g$,  and $\displaystyle \lim_{a \rightarrow \infty} \overline{K}_g(a)=\overline{K}_g$. 

We also note the following property which shows, in particular that the pinching constants $\delta$ and $\delta(a)$ mentioned in the introduction are less than or equal to $1$. 

\begin{proposition}
Let $(M,g)$ be a compact manifold then $\underline{\kappa}_g(a) \leq \overline{K}_g(a)$ for all $a \geq 0$. 
\end{proposition}

\begin{proof}
We first claim that there exists a real number $\kappa$ arbitrarily close to $\underline\kappa_g(a)$ and a function $\phi_1$ such that $\overline\sec_{\phi_1} \geq \kappa$. Indeed, if $(M,g)$ has PWSC, then we may chose $\kappa > 0$ and $\phi_1 \leq 0$, and these imply that $\kappa e^{-4\phi_1} \geq \kappa$. If $(M,g)$ does not have PWSC, then we may choose $\kappa \leq 0$ and $\phi_1 \geq 0$ and again conclude that $\kappa e^{-4\phi_1} \geq \kappa$. Second, a similar argument shows that there exist a real number $K$ arbitrarily close to $\overline K_g(a)$ and a function $\phi_2$ such that $\overline\sec_{\phi_2} \leq K$.

Now assume that $\underline\kappa_g(a)  >  \overline K_g(a)$. By the previous paragraph, we may choose $\kappa > K$ and functions $\phi_1$ and $\phi_2$ such that $\overline\sec_{\phi_1} \geq \kappa > K \geq \overline\sec_{\phi_2}$. Subtracting, we obtain the inequality
\begin{align} \label{eqn:PosHess} \mathrm{Hess} (\phi_1-\phi_2) (U,U) + d\phi_1(U)^2 - d\phi_2(U)^2 > 0.
\end{align}
for all unit vectors $U$. Since $M$ is compact, the function $\phi_1 - \phi_2$ achieves a maximum at some point $p \in M$. Since $d\phi_1 = d\phi_2$ and $\Hess(\phi_1 - \phi_2) \leq 0$ at $p$, this is a contradiction.
\end{proof}

\begin{remark} \label{Rem:Zero} A final simple remark about the definitions of $\underline{\kappa}_g$ and $\overline{K}_g$ that comes from (\ref{eqn:AddConstant}) is that we can always assume that our density is normalized so that $\phi(p) = 0$ for some $p \in M$.  This is because if there is not a zero, a constant can be added to the density to give it one, improve the curvature bound,  and preserve  $\phi \leq 0$ or $\phi \geq 0$. 
\end{remark}

\subsection{Examples} 

In this section we discuss some basic examples of what our results and the  earlier results of \cite{KennardWylie17, WylieYeroshkin,Wylie16,Wylie} tell us about PWSC, NWSC,  $\underline{\kappa}, \overline{K},$ and $\delta$.   To organize the exposition in this section, we ask the following question. 

\begin{moquestion} Let $(M,g)$ be a compact Riemannian manifold. If $(M,g)$ has PSWC (NWSC) is there another metric $\widehat{g}$ on $M$ such that $\sec_{\widehat{g}}>0 (<0)$? Is there a metric, $\widehat{g}$, on $M$ such that $\sec^{\widehat{g}}_{\min} = \underline{\kappa}_g$,  $\sec^{\widehat{g}}_{\max} = \overline{K}_g$, or $\frac{\sec^{\widehat{g}}_{\min}}{\sec^{\widehat{g}}_{\max}} = \delta_g$? 
\end{moquestion}

All the results of this paper can be organized as partial results towards understanding this question.  For example,  Theorem \ref{Thm:QuarterPinch} shows the answer is yes up to homeomorphism when $\delta_g> 1/4$.  There is also simple answer to this question when the metric $g$ is locally homogeneous. 

\begin{proposition} \label{Prop:LocallyHomogeneous}
Let $(M,g)$ be a compact locally homogeneous space.  Then $\underline{\kappa} = \sec_{min}$ and $\overline{K} = \sec_{max}$.
\end{proposition} 
\begin{proof}
Let $\phi$ be an arbitrary function on $(M,g)$. At a maximum of $\phi$,  $\sec_{\phi}(U,V) \leq  \sec(U,V)$ for all $U,V$.  Similarly, at a minimum of $\phi$, $\sec_{\phi}(U,V) \geq \sec(U,V)$.  Since the sectional curvatures do not depend on the point, this implies the proposition. 
\end{proof}

\begin{remark} Spaces of constant curvature and  symmetric spaces with their canonical metrics are locally homogeneous, so Proposition \ref{Prop:LocallyHomogeneous} applies to all of these nice spaces.
\end{remark} 

Explicit examples of metrics with  PWSC but $\sec^g_{min}<0$ are constructed in \cite[Propositions 2.11 \& 2.16]{KennardWylie17}.  These metrics are rotationally symmetric metrics on the sphere and cohomogeneity one metrics on $\mathbb{C}P^n$.  These examples show that the space of metrics with PWSC is larger than the space of metrics with positive sectional curvature, but  does not address the question of whether there are topologies which support PWSC but not positive sectional curvature.  On the other hand,  in dimensions 2 and 3 a compact manifold has PWSC if and only if there is a metric on $M$ with positive sectional curvature. This follows from that fact that $\pi_1(M)$ must be finite \cite[Theorem 1.6]{Wylie15},  along with the Gauss-Bonnet Theorem and geometrization of $3$-manifolds.  
 
 In the case of non-positive curvature, there is a weighted Cartan-Hadamard Theorem  \cite[Theorem 1.2]{Wylie15} which implies that if there is a function such that $\overline{\sec}_{\phi} \leq 0$ then the metric has no conjugate points and thus $M$ must be a $K(\pi,1)$ space.  This combined with the Myers' theorem shows that a given compact manifold $M$ can not admit separate metrics with PWSC and $\overline{K} \leq 0$.
 
There is also a Cheeger-Gromoll type splitting theorem for the condition $\overline{\sec}_{\phi} \geq 0$, \cite[Theorem 6.3]{Wylie}.  The statement of this result is complicated somewhat by a loss of rigidity in the conclusion to a warped product splitting instead of the traditional direct product  as well as necessary boundedness conditions on the density.  However,  the classical  topological obstructions to a compact manifold admitting a metric of non-negative sectional curvature that $b_1(M) \leq n$ with equality only if it is flat and $\pi_1(M)$ has an abelian subgroup of finite index are shown to be true for $\overline{\sec}_{\phi} \geq 0$ \cite[Theorem 1.5]{Wylie}.  
 
 These previous results, along with a deep result of Burago and Ivanov, we  obtain the following information for the torus. 

\begin{example}  Let $(T^n, g)$  be any Riemannian metric on a torus.   Then by the weighted Myers' theorem $g$ does not have PWSC, and by Byers' theorem it does not have NWSC.    Moreover, by the splitting theorem there is a density with $\overline{\sec}_{\phi} \geq 0$  if and only if $g$ is a flat metric.  On the other hand, by the weighted Cartan-Hadamard Theorem if there is a density with $\overline{\sec}_{\phi} \leq 0$  then the metric has no conjugate points.   Burago and Ivanov \cite{BuragoIvanov94} have proven  that a metric on the torus with out conjugate points must be flat. Therefore,  there  is a density with $\overline{\sec}_{\phi} \leq 0$  if and only if $g$ is a flat metric \end{example}

  Using Theorem \ref{Thm:Byers}, we can generalize the torus example to any manifold admitting a flat metric. 

\begin{example} Let $M^n$ be a compact manifold which admits a flat metric.  By the first Bieberbach theorem, $\pi_1(M)$ contains a free abelian group on $n$-generators.  Therefore, by Theorem \ref{Thm:Myers}, the manifold does not admit PWSC and by Theorem \ref{Thm:Byers}, it does not admit NWSC. \end{example}

This example along with the Myers' and Cartan Hadamard Theorems shows that for a compact surface  the topologies that admit PWSC, NWSC, $\overline{\sec}_{\phi} \geq 0$ or $\overline{\sec}_{\phi} \leq 0$ are all equivalent to the standard topologies admitting the corresponding unweighted curvature conditions.

Another well known application of Theorem \ref{Thm:Byers} is the following. 

\begin{example} Let $M_1$, $M_2$ be compact manifolds then $M_1 \times M_2$ does not admit NWSC.  If it did, then by the Weighted Cartan-Hadamard Theorem, $\pi_1(M_1)$ and $\pi_1(M_2)$ must both be infinite.  Then, taking one generator in each factor of $\pi_1(M_1 \times M_2) = \pi_1(M_1) \times \pi_1(M_2)$ gives an abelian subgroup which is not cyclic, contradicting Theorem  \ref{Thm:Byers}. 
\end{example}

On the other hand, the question of whether $M_1 \times M_2$ can admit PWSC is a difficult question, which is a generalization of the famous Hopf conjecture that $S^2 \times S^2$ does not admit a metric of positive sectional curvature.   

We also note that  totally geodesic submanifolds can obstruct improving the curvature by adding a density. 

 \begin{proposition} \label{Prop:TotGeodTorus}
 Let $(M,g)$ be a complete Riemannian manifold.  Let $(N,h)$ be a compact, totally geodesic submanifold.  Then $(N,h)$ must contain points with $p$ and $q$ with  $\sec_N(p) \geq \underline{\kappa}$ and $\sec_N(q) \leq \overline{K}$. In particular, a metric admitting a totally geodesic flat torus can not have PWSC nor NWSC.  \end{proposition}
 
 \begin{proof}
Since $N$ is totally geodesic the Hessian on $(N,h)$ of the restriction of $\phi$  to $N$ is equal to the restriction of  $\mathrm{Hess}_g \phi$ to $T_pN$. $\phi$ restricted to $N$ has a maximum and minimum as $N$ is compact.  Let $p$ be a local maximum of $\phi$ restricted to $N$, and let $U,V \in T_pN$.  Then
\begin{align*}
 \overline{\sec}_{\phi}(U,V) &= \sec_M(U,V) + \mathrm{Hess} \phi (U,U) + d\phi(U)^2 \\
 &\leq \sec_N(U,V)
 \end{align*}
 Similarly, if we take $q$ to be a local minimum of $\phi$ restricted to $N$, we obtain the opposite inequality. 
 \end{proof}

There are a number of metrics with non-negative sectional curvature which have totally geodesic flat tori.  For example,  in \cite{Wilhelm01} Wilhelm constructs metrics on the Gromoll-Meyer sphere with $\sec \geq 0$ and $\sec>0$ almost everywhere, which contain flat tori.  Proposition \ref{Prop:TotGeodTorus} shows that these metrics do not have PWSC.  Furthermore, in \cite{Wilking02} Wilking showed that any normal biquotient that has some flat planes must have an immersed totally geodesic flat submanifold, usually this is a torus.  Also see \cite{EschenburgKerin08, PetersenWilhelm-pre}.

\subsection{Non-compact case}

While the main focus of this paper is compact manifolds, and Definition \ref{Def:CurveBound} is intended mainly for the compact case, we include a few remarks here in the non-compact case for completeness.

 Note that Theorem \ref{Thm:Myers} shows that $\underline{\kappa}_g =0$ for any noncompact manifold with PWSC. On the other hand, there are a number of simple examples of noncompact metrics with PWSC which do not have positive curvature.  First we consider the flat Euclidean space

\begin{example} \label{Ex:Euclidean}
Consider the $\mathbb{R}^n$ with the flat metric.  Let $\phi (x)= \frac{\kappa}{2} |x|^2$  then \begin{align*}
\overline{\sec}_{\phi}(U,V) &= \mathrm{Hess} \phi(U,U) + d\phi(U)^2 \\ 
&= \kappa  + d\phi(U)^2 \\
& \geq \kappa e^{-4\phi}. 
\end{align*} 
This shows that $\mathbb{R}^n$ has PWSC, however the density $\phi$ is not bounded above.   In fact, in this case it is easy to see directly that  $\underline{\kappa} =0$.  Let $\phi$ be a function such that $\overline{\sec}_{\phi} >0$, then restricting $\phi$ along a geodesic we have $\phi'' + (\phi')^2 >0$. Set $u = e^{\phi}$. Then we have $u''>0$.  However, if $\phi$ is bounded above then so is $u$.  This is not possible if $u$ is defined along the whole line.  

Similarly,  $\mathbb{R}^n$ has $\overline{K}_g = 0$.  In fact, there is no density $\phi$, bounded or not, such that $\overline{\sec}_{\phi}<0$.  To see this suppose $\phi$ were such a density. Then, restricting $\phi$ to a geodesic, we'd have a non-constant function of $1$-variable defined on the entire real line such that $\phi'' + \phi^2 <0$.  Then the function $u = e^{\phi}$ satisfies $u''\leq 0$ and $u>0$, which is not possible.  
\end{example}

Generalizing this example, any Cartan-Hadamard space of bounded curvature has PWSC. 

\begin{example}  \label{Ex:CH}
Let $(M,g)$ be a simply connected manifold with $A \leq \sec_g \leq 0$.  Pick a point $p$ and let $\phi(x) = \frac{B}{2} r^2_p$ where $r_p$ is the distance function to $p$.  Then $\mathrm{Hess} \phi \geq B g$, so if $B>A$ then $\sec_{\phi} >0$.   On the other hand,  these metrics all have lines,  so by the splitting theorem there is no function $\phi$  bounded above such that $\overline{\sec}_{\phi} >0$. 
\end{example}

If we dot not assume a bound on $\phi$, the next example shows in the non-compact case that there are topologies which support PWSC but have no metric of positive sectional curvature. 

\begin{example}  \label{Ex:RN} \cite[Proposition 2.8]{KennardWylie17}  gives   metrics on $\mathbb{R} \times N$ where $N$ is a manifold admitting non-negative sectional with PWSC. The metrics are simple warped products of the form 
\[ g = dr^2 + e^{2r} g_N \qquad \phi = Ar.  \]   Moreover, if $N$ is compact then $\mathbb{R} \times N$ has two ends, so by the splitting theorem they can not admit $\sec_{\phi} >0$ for any $\phi$ which is bounded above. 
 \end{example}
 
These examples indicate that there should be many examples of non-compact spaces of PWSC, if one does not make any assumptions about the function $\phi$.   On the other hand, Theorem \cite[Theorem 2.9]{WylieYeroshkin} shows that if a complete Riemannian manifold  admits a function $\phi$ such that $\overline{\sec}_{\phi} \geq \kappa e^{-4\phi}$ for some $\kappa>0$ then $\pi_1(M)$ is finite.  Thus, while Examples \ref{Ex:Euclidean} and \ref{Ex:CH} admit such densities, the manifolds in Example \ref{Ex:RN} do not in general. 

In the non-compact case it would also be interesting to study bounds on $\overline{\sec}_{\phi}$ for other  asymptotics of the density $\phi$ besides being bounded. 
\section{Weighted Convexity}
\label{sec:Convexity}
\subsection{Preliminaries}

Sectional curvature bounds give control of the Hessian of the distance function, which imply convexity properties of the underlying metric space.  In order to see what kind of convexity is implied by weighted sectional curvature bounds we consider the Hessian under a conformal change. Given $(M,g, \phi)$ let $\widetilde{g} = e^{-2\phi}g$.  Recall that for a smooth function $u$, the formula relating the Hessian in $g$ and  $\widetilde{g}$ is
\begin{align} \label{Eqn:ConfHess}
\mathrm{Hess}_{\widetilde{g}} u = \mathrm{Hess}_{g} u + d\phi \otimes du + du\otimes d\phi - g(\nabla \phi, \nabla u)g. 
\end{align}
Consider a distance function $r$ for the metric $g$ and take its Hessian with respect to the conformal metric $\widetilde{g}$. The orthogonal complement of the gradient is well defined in a conformal class since conformal change preserves angle and  modifies the gradient by a scalar factor.  Consider vectors $U,V \perp \nabla r$,  then we have 
\begin{align} \label{Eqn:PerpConfHess}
\mathrm{Hess}_{\widetilde{g}} r(U,V) = \mathrm{Hess}_g r(U,V) - g(\nabla \phi, \nabla r)g(U,V).
\end{align}
Geometrically, up to multiplying by a suitable factor of $e^{\phi}$, $\mathrm{Hess}_{\widetilde{g}} r(U,V)$ represents the second fundamental form with respect to the conformal metric  of the level sets of $r$.  We will develop the tools which allow us to control this quantity from bounds on the curvature $\overline{\sec}_{\phi}$ in an analogous way that the classical sectional curvature control Hessian of the distance function. 

In applying these results, we encounter a technical issue not present in the un-weighted setting. Namely, $\nabla r$ is a null vector for $\mathrm{Hess}_g r$, but we can see from (\ref{Eqn:ConfHess}) that this is not true for $\mathrm{Hess}_{\widetilde g} r$ as,  if $U \perp \nabla r$, then
\begin{align*}
\mathrm{Hess}_{\widetilde g} r (U, \nabla r) &= d\phi(U) \\
\mathrm{Hess}_{\widetilde g} r (\nabla r, \nabla r) &= d\phi(\nabla r). 
\end{align*}
Therefore, $\nabla r$ is an eigenvector for $\mathrm{Hess}_{\widetilde g} r$ if and only if $\phi$ is a function of $r$, and is a null vector if and only if $\phi$ is constant. In order to get around this, we will have to consider a lower order perturbation of $\mathrm{Hess}_{\widetilde{g}} r$, namely, from (\ref{Eqn:ConfHess}) we have that
\begin{align*}
 \left( \mathrm{Hess}_{\widetilde{g}} r - d\phi \otimes dr - dr\otimes d\phi \right)(U, \nabla r) &= 0 \\
 \left( \mathrm{Hess}_{\widetilde{g}} r - d\phi \otimes dr - dr\otimes d\phi \right)(\nabla r, \nabla r) &= -d\phi(\nabla r). 
\end{align*}
Therefore, $\nabla r$ is at least an eigenvector for the modified Hessian $\mathrm{Hess}_{\widetilde{g}} r - d\phi \otimes dr - dr\otimes d\phi$.   Moreover, the modified Hessian has nice convexity properties along geodesics.  Namely, of $\widetilde{\sigma}$ is a geodesic for $\widetilde{g}$ and $u$ is a smooth function then 
\begin{align}
\left( \mathrm{Hess}_{\widetilde{g}} u - d\phi \otimes du - du\otimes d\phi \right)(\widetilde{\sigma}', \widetilde{\sigma}')  = u'' - 2\phi'u'.
\end{align}
We will have to keep in mind below that $\nabla r$ is not a null vector for our modified conformal Hessian.  We will see in the next section that it is not hard to overcome this problem by using modified distance functions.  However it has the unpleasant effect of making our modified distance functions an abstract solution to an ODE involving $\phi$ instead of the simple explicit functions used in the non-weighted setting.  

\subsubsection{Modified Hessian and the weighted connection }

Now we discuss the relationship between the weighted connection $\nabla^{\phi}$ and conformal Hessian. 
The Riemannian Hessian can be expressed in terms of the Levi-Civita connection in the following two ways. 
\begin{align}
\mathrm{Hess} u(U, V) &= g(\nabla_U \nabla u, V ) \label{Hess2} \\
&= (\nabla_U du)(V)  \label{Hess3}
\end{align}

  On the other hand, if we replace  the Levi-Civita connection by the weighted connection $\nabla^{\phi}$ in (\ref{Hess2}) and (\ref{Hess3}) we get two different tensors. 
\begin{align}
g(\nabla^\phi_U \nabla u, V) &= g(\nabla_U \nabla u, V) - d\phi(U)g(\nabla u, V) - d\phi(\nabla u) g(U,V) \nonumber\\
&= \Hess u(U,V) - d\phi(U)du(V) - d\phi(\nabla u)g(U,V). \label{WHess2}\\
\left(\nabla^\phi_U du\right)(V)&= D_U du(V) - du(\nabla^\phi_U V) \nonumber\\
&= D_U du(V) - du(\nabla_U V) + d\phi(U)du(V) + d\phi(V)du(U) \nonumber\\
&= \Hess u(U,V) + d\phi(U)du(V) + d\phi(V) du(U) \label{WHess3}
\end{align}

These two Hessians are different exactly because the connection $\nabla^{\phi}$ is not compatible with the metric. Note also that (\ref{WHess2}) is not symmetric in $U$ and $V$, while (\ref{WHess3}) is.  To see the relation to the conformal Hessian, note that combining (\ref{WHess2}) and (\ref{Eqn:PerpConfHess}) for $U,V \perp \nabla u$ we have 
\begin{align}
\mathrm{Hess}_{\widetilde{g}} u (U,V) = g(\nabla^\phi_U \nabla u, V).  \label{Eqn:MetricHessian}
\end{align}
Moreover, we can see that the modified conformal Hessian we saw in the previous section is related to (\ref{WHess3}) via the formula
\begin{align}
\left(\nabla^{\widetilde{g}, -\phi}_{\cdot} du\right)(\cdot) = \mathrm{Hess}_{\widetilde{g}} u - d\phi \otimes du  - du\otimes d\phi. \label{Eqn:IntrinsicHessian}
\end{align}
Where  $\nabla^{\widetilde{g}, -\phi}$ is the weighted connection for the metric $\widetilde{g}$ with density $-\phi$.   

The conformal change $(g, \phi) \rightarrow (\widetilde{g}, -\phi)$ also has natural curvature properties as it has been observed in \cite{Wylie15} that the sign of the curvature $\overline{\sec}_{g, \phi}$ is the same as the sign of the curvature  $\overline{\sec}_{\widetilde{g}, -\phi}$.  Thus the operation  $(g, \phi) \rightarrow (\widetilde{g}, -\phi)$ is an involution on the space of metrics with density that preserves the conditions of positive and negative weighted sectional curvature. 

While equations (\ref{Eqn:MetricHessian}) and (\ref{Eqn:IntrinsicHessian}) will not be explicitly used in the proofs of our comparison theorems, abstractly they explain why the curvatures coming from the weighted connection $\nabla^{\phi}$ should control the conformal Hessian of the distance function. 

\subsection{Non-positive curvature}

Now we consider Riemannian manifolds $(M,g)$ which admit a density $\phi$ such that $\overline{\sec}_{\phi} \leq 0$. In this case we initially  do not need to make any boundedness assumptions on the density for results.  It was proven in \cite{Wylie15} that  if $\overline{\sec}_{\phi} \leq 0$ then the metric does not have conjugate points.   This follows from the following set of  formulas, derived in the proof of Theorem 4.2 in \cite{Wylie15}, which  we will also find useful.  Let $\sigma(t)$ be a unit speed geodesic and $J(t)$ a perpendicular Jacobi field along $\sigma$.  Then we have 
\begin{align} \label{Eqn:JacobiNonPos} \begin{split}
\frac{d}{dt}\left( \frac{1}{2}e^{-2\phi}|J|^2 \right) &= e^{-2\phi}g(J' - d\phi(\sigma')J,J) \\
\frac{d}{dt} g(J' - d\phi(\sigma')J,J) &\geq |J' - d\phi(\sigma')J|^2 \geq 0.  \end{split}
\end{align} 
From which it follows that if $J(0) = 0$ then $\frac{d}{dt}\frac{1}{2}e^{-2\phi}|J|^2 \geq 0$ and if $J(t_0) = 0$ then $J(t) = 0$ for all $t \in [0, t_0]$. 

We also recall the second variation of energy formula.  Given a variation $\overline \sigma:[a,b]\times(-\varepsilon,\varepsilon) \to M$ of a geodesic $\sigma = \overline\sigma(\cdot,0)$, let $V = \left.\frac{\partial\overline\sigma}{\partial s}\right|_{s=0}$ denote the variation vector field along $\sigma$. The second variation of energy is given by
	\[\left.\frac{d^2}{ds^2}\right|_{s=0} E(\sigma_s)
	= I(V,V) + \left.g\of{\frac{\partial^2\overline\sigma}{\partial s^2}, \frac{\partial \overline{\sigma}}{\partial t}}\right|_{a}^{b},\]
where $I(V,V)$ is the index form of $\sigma$.  The usual formula for the index form is 
\[ I(V,V) = \int_a^b \of{ |V'|^2 -g(R(V,\sigma') \sigma', V) } dt .\]
When $V$ is perpendicular to $\sigma$ the index form can be re-written as follows (see \cite[Section 5]{Wylie15}):
	\begin{eqnarray}
	\hspace{.1in}I(V,V)\hspace{-.1in} &=& \hspace{-.1in} \int_a^b \of{ |V' - d\phi(\sigma')V|^2 
	- g(R^{\nabla^{\phi}}(V,\sigma') \sigma', V)}dt
	+ \left.d\phi(\sigma') |V|^2\right|_{a}^{b}\label{eqn:IndexForm}
	\end{eqnarray}
Using this  formula we obtain positivity of the conformal Hessian of the distance function when applied to vectors orthogonal to the gradient.

\begin{lemma} \label{Lem:NonNegConvex}
Suppose that $(M,g)$ is a simply connected complete manifold with density, $\phi$, such that $\overline{\sec}_{\phi} \leq 0$.  Then for any point $p \in M$, 
\begin{align}
\mathrm{Hess}_{\widetilde{g}} \left( \frac{1}{2} r_p^2 \right) (U,U) > 0 \qquad \forall U \perp \nabla r_p 
\end{align}
where $\widetilde{g} = e^{-2\phi} g$ and $r_p(\cdot) = d^g(p, \cdot)$ is the distance function for the $g$-distance.
\end{lemma}

\begin{proof}
For a vector $U$ based at a point $q$ and perpendicular to $\nabla r_p$ let $\sigma(t)$ be the minimizing $g$-geodesic from $p$ to $q$ and $\widetilde{\sigma}(s)$ be the $\widetilde{g}$-geodesic with $\widetilde{\sigma}(0) = q$ and $\widetilde{\sigma}'(0) = U$.  Let $\overline{\sigma}:[0,1] \times (-\varepsilon, \varepsilon) \rightarrow M$ be the variation constructed so that the curve $t \rightarrow \overline{\sigma}(t, s_0)$ is the unique minimizing $g$-geodesic from $p$ to $\widetilde{\sigma}(s_0)$.  $\overline{\sigma}$ is an orthogonal variation of the geodesic $\sigma$,  the variation field is a $g$-Jacobi field, $J$, and $\left.\frac{d^2}{ds^2}\right|_{s=0} E(\sigma_s)=\mathrm{Hess}_{\widetilde{g}} \left( \frac{1}{2} r_p^2 \right) (U,U) $. 

From (\ref{eqn:IndexForm}) we have
\begin{align*}
\left.\frac{d^2}{ds^2}\right|_{s=0} E(\sigma_s)
	&= \int_0^1 \of{ |J' - d\phi(\sigma')J|^2  - g(R^{\nabla^{\phi}}(J,\sigma') \sigma', J)}dt\\
	&\quad + \left.d\phi(\sigma')|J|^2\right|_{t=0}^{t=1}+ \left.g\of{\frac{\partial^2\overline\sigma}{\partial s^2}, \frac{\partial \overline{\sigma}}{\partial t}}\right|_{t=0}^{t=1}, 
\end{align*}

By (\ref{Eqn:JacobiNonPos}) $\int_0^1 |J' - d\phi(\sigma')J|^2 dt > 0$.  So, 
\begin{align*}
\left.\frac{d^2}{ds^2}\right|_{s=0} E(\sigma_s)
	&>d\phi(\sigma'(1)) \left|\frac{\partial \widetilde{\sigma}}{\partial s} \right|_g^2 + g\left(\nabla^g_{\frac{\partial \widetilde{\sigma}}{ \partial s}} \frac{\partial \widetilde{\sigma}}{ \partial s},  \sigma'(1)\right) .
\end{align*}
Recall that the formula for the Levi-Civita connection of $\widetilde{g}$ is 
  \begin{align}
  \widetilde{\nabla}_X Y = \nabla_X Y - d\phi(X) Y - d\phi(Y)X + g(X,Y) \nabla \phi. 
  \end{align}
  So $\widetilde{\sigma}(s)$ is a $\widetilde{g}$-geodesic implies that 
  \begin{align*}
  g\left(\nabla^g_{\frac{\partial \widetilde{\sigma}}{ \partial s}} \frac{\partial \widetilde{\sigma}}{ \partial s},  \sigma'(1)\right) + d\phi(\sigma'(1)) \left|\frac{\partial \widetilde{\sigma}}{\partial s} \right|_g^2=0.
  \end{align*}
  So  $\left.\frac{d^2}{ds^2}\right|_{s=0} E(\sigma_s) > 0$.
  \end{proof}
  
Now let $r$  be the distance to a closed subset $A$, $r(x) = d(x,A)$. $r$ is smooth on an open dense subset of $M \setminus A$, and on the set we can write the metric as $g = dr^2 + g_r$ where $g_r$ is a family of metric on the level sets of $r$.   We will say a function $u$ is a \emph{modified distance function to $A$} if there is a smooth function $h: [0, \infty) \rightarrow [0, \infty)$ with $h(0) = h'(0)=0$ and $h'(r)>0$ for $r>0$ such that $u= h\circ r$. For example $\frac{1}{2} r^2$ is a modified distance function. We have the following formula for the modified Hessian of a modified distance function.

\begin{proposition}  \label{HessianModifiedDistance}
Let $u$ be a modified distance function.  At points where $u$ is smooth, 
\begin{align*}
\mathrm{Hess}_{\widetilde{g}} u - d\phi \otimes du - du \otimes d \phi &= \left( h''  -h' \frac{\partial \phi}{\partial r}\right) dr \otimes dr +h' \left( \mathrm{Hess}_g r - g(\nabla r, \nabla \phi) g_r \right).
\end{align*}
\end{proposition}

\begin{proof}	
A standard formula for $\mathrm{Hess}_g u$ is 
\begin{align*}
\mathrm{Hess}_{g}  u &= h'' dr \otimes dr + h' \mathrm{Hess}_g r 
\end{align*}
Combining this with the formula for the conformal Hessian (\ref{Eqn:ConfHess}) gives
\begin{align*}
& \mathrm{Hess}_{\widetilde{g}} u - d\phi \otimes du - du \otimes d \phi  \nonumber \\ & \qquad = \mathrm{Hess}_g u  - g(\nabla u, \nabla \phi) g \nonumber\\
& \qquad = \left( h''  -h' \frac{\partial \phi}{\partial r}\right) dr \otimes dr +h' \left( \mathrm{Hess}_g r - g(\nabla r, \nabla \phi) g_r \right) 
\end{align*}
\end{proof}

Proposition \ref{HessianModifiedDistance} combined with Lemma \ref{Lem:NonNegConvex} gives us the following. 

\begin{theorem} \label{Thm:ModifiedDistanceConvex}
Suppose that $(M,g, \phi)$ is a simply connected complete manifold with density such that $\overline{\sec}_{\phi} \leq 0$.  Then for any point $p \in M$, there is a modified distance function to $p$,  $u_p$,   such that 
\begin{align}
\mathrm{Hess}_{\widetilde{g}} u_p - d\phi \otimes du_p - du_p \otimes d \phi &> 0.
\end{align}
\end{theorem}

\begin{proof}
By  Lemma \ref{Lem:NonNegConvex} and Proposition \ref{HessianModifiedDistance},  $\mathrm{Hess}_g r - g(\nabla r, \nabla \phi) g_r > 0$ on the orthogonal complement to $\nabla r$. Let $a:[0, \infty) \rightarrow [0, \infty)$ be a smooth function  such that $|d\phi_{q}| < a(r)$ for all $q \in B(p,r).$ Such a function exists by the compactness of $B(p,r)$.  Then define $u_p= h \circ r_p$ where $h$ is the solution to be the solution to $h''-h'a = 1$, $h(0) = 0$, $h'(0) = 0$.  Since
\begin{align*}
0< e^{\int a}\left( h'' - h'a\right) = (h' e^{\int a})', 
\end{align*}
$h'(r) >0$ for $r>0$, so $h$ is a modified distance function.  Then  $h'' - h' \frac{\partial \phi}{\partial r}\geq h'' - h'a=1$. So by Proposition \ref{HessianModifiedDistance}, the theorem follows.   
\end{proof}

Recall the result of Cartan that an isometry of finite order of a space with nonpositive curvature must have a fixed point. 
We generalize this to manifolds with density under the extra technical assumption that $\widetilde{g}$ is complete. Note that this condition is satisfied for the universal cover of a compact space with $\overline{\sec}_{\phi} \leq 0$. 

We will call a function $u$ such that $\mathrm{Hess}_{\widetilde{g}} u - d\phi \otimes du - du \otimes d \phi > 0$ a weighted strictly convex function (with respect to $(g, \phi)$).  For such a function along a $\widetilde{g}$-geodesic  $\widetilde{\sigma}(t)$ we have 
\begin{align}  \label{Eqn:WeightConvex}
(u \circ \widetilde{\sigma}) '' -2 \phi' (u \circ \widetilde{\sigma})' > 0.
\end{align}
Letting $s$ be the new parameter along $\widetilde{\sigma}$ such that $ds = e^{2\phi} dt$ we can see that (\ref{Eqn:WeightConvex}) is equivalent to 
\begin{align*}
\frac{d^2}{ds^2} \left( u \circ \widetilde{\sigma}\right) > 0.
\end{align*}
Thus the restriction of $u$ to $\widetilde{g}$-geodesics is convex in the $s$ parameter.  Since $s$ only depends on the metric $g$ and function $\phi$ we can apply standard results from  the theory of $1$-dimensional convex function to weighted convex functions.  For example, it follows that the maximum of a finite collection of strictly weighted convex functions  is strictly weighted convex and  if $\widetilde{g}$ is complete then any proper, nonnegative and strictly weighted convex function has a unique minimum.

Now we can modify the construction in Theorem \ref{Thm:ModifiedDistanceConvex} slightly to define a weighted notion of $L^{\infty}$ center of mass when the metric $\widetilde{g}$ is complete.   For a finite collection of points $p_1, \dots p_k$ let $a(r)$ be a smooth function such that $|d\phi|_{q} \leq a(r)$ for all $q \in \cup_{i=1}^k B(p_i, r)$ and let $h$ be the function solving $h''-ah'=1$, $h(0)=0$, $h'(0)=0$.  Then by the same argument as in the proof of Theorem \ref{Thm:ModifiedDistanceConvex}, the function $h(r_{p_i})$ is weighted strictly convex. Define $u_{p_1,\dots p_k} = \max \{ h(r_{p_1}), \dots, h(r_{p_k}) \}$.  Then we define the $L^{\infty}$ weighted center of mass  of $\{p_1, \dots p_k\}$, $cm_{\infty}^{\phi}\{ p_1, p_2, \dots, p_k \}$ as the unique maximum of $u_{p_1, \dots, p_k}$.  This notion allows us to generalize the proof  of Cartan. 

\begin{theorem}
Suppose that $(M,g, \phi)$ is a simply connected complete manifold with density such that $\overline{\sec}_{\phi} \leq 0$.  Suppose in addition that the metric $\widetilde{g}$ is complete, then any isometry of finite order has a fixed point. 
\end{theorem}

\begin{proof}
Let $F$ be an isometry of $g$ and let $k$ be the order of $F$.  For any $p\in M$, let $q = cm_{\infty}^{\phi} \{p, F(p), F^2(p), \dots F^{k-1}(p)\}$. We claim that $q$ is a fixed point. Since $F$ is an isometry we have
\begin{align*}
u_{p, F(p), \dots, F^{k-1}(p)} (F(q)) & = \max \left\{ h(d(p, F(q)), h(d(F(p), F(q)), \dots h(d(F^{k-1}(p), F(q)) \right\}\\
&= \max \left\{h(d(F^{k-1}(p), q), h(d(p,q)), \dots, h(d(F^{k-2}(p), q) \right\} \\
&=u_{p, F(p), \dots, F^{k-1}(p)}(q).
\end{align*}
Since $q$ is the unique maximum, $F(q) = q$. 
\end{proof}

Recall that manifolds with non-positive sectional curvature not only have no conjugate points, but also satisfy the stronger condition of having no focal points, meaning that any geodesic does not have focal points. We have the following modification of this property in terms of the conformal change $\widetilde{g}$ for $\overline{\sec}_{\phi} \leq 0$. 

\begin{lemma} Suppose that $(M,g)$ is a Riemannian manifold  admitting a density such that $\overline{\sec}_{\phi} \leq 0$.  Let $H$ be a totally geodesic submanifold for the metric $\widetilde{g}$, then $H$ has no focal points with respect to the $g$-metric.  If, in addition, $M$ is simply connected, then the normal exponential map of $H$ in the metric $g$,  $exp^{\perp}: \nu(H) \rightarrow M$, is a diffeomorphism. 
\end{lemma}

\begin{proof}
First we show that $H$ does not have focal points. Let $\sigma$ be a $g$-geodesic with $p=\sigma(0)\in H$ and $\sigma'(0) \perp H$.  A  Jacobi field $J$ along $\sigma$   is called an  $H$-Jacobi field  if it satisfies $J(0)\in T_p H$ and $J'(0) - S_{\sigma'(0)} (J(0)) \in (T_pH)^{\perp}$, where $S_{N}(X) = ( \nabla_X N)^{T}$ is the second fundamental form of $g$ with respect to the normal vector $N$.   $\sigma(t_0)$ is a focal point of $H$ if there is an $H$-Jacobi field along $\sigma$ with $J(t_0)= 0$.  The second fundamental form of $H$ with respect to  $\widetilde{g}$  is given by 
\begin{align*}
(\nabla^{\widetilde{g}}_X N)^{T} &= \left( \nabla_X N - d\phi(X) N - d\phi(N)X + g(X,N) \nabla \phi\right)^{T}\\
&= S_{N} X - d\phi(N) X
\end{align*}
Therefore, if $H$ it $\widetilde{g}$-totally geodesic, an $H$-Jacobi field satisfies $J(0) \in T_pH$ and $J'(0) - d\phi(\sigma'(0)) J(0) \in (T_pH)^{\perp}$.  In particular, $g(J'(0) - d\phi(\sigma'(0))J(0),J(0)) = 0$.  Then from (\ref{Eqn:JacobiNonPos}) $\frac{d}{dt}\left( \frac{1}{2}e^{-2\phi}|J|^2 \right) \geq 0$ for any $H$-Jacobi field and thus can never vanish.  Therefore $H$ does not have focal points and the normal exponential map is a local diffeomorphism. 

Now we have to  show that $exp^{\perp}$ is one to one when $M$ is simply connected.  Suppose not.  Then there is a point $p$ and two minimizing geodesics from $p$ to $H$ that minimize the distance from $p$ to $H$. By the weighted Cartan-Hadamard theorem, these two geodesics must hit different points on $H$, call them $a$ and $b$.   Let $\widetilde{\gamma}$ be the $\widetilde{g}$-geodesic connecting $a$ to $b$ which must lie on $H$.   Then consider the modified distance function to $p$, $u = h \circ r_p$ and its restriction to $\widetilde{\gamma}$,  $\widetilde{u} = (u \circ \widetilde{\gamma})$.  We have
\begin{align*}
\widetilde{u}'' - 2 \phi' \widetilde{u}' > 0.  
\end{align*}
Thus  $(e^{-2\phi} \widetilde{u}')' >0$ so that 
\begin{align*}
e^{-2\phi(\widetilde{\gamma}(t))} \widetilde{u}'(t) > e^{-2\phi(\widetilde{\gamma}(0))} \widetilde{u}'(0) 
\end{align*}
However, $u' = g(\nabla u, \widetilde{\gamma}') = h' g(\nabla r, \widetilde{\gamma}')$ which is zero on each endpoint, since the minimal geodesics from $p$ to $a$ and $b$ meet $H$ perpendicularly and $\widetilde{\gamma}$ is on $H$.  This gives a contradiction. 
\end{proof}

This lemma now tells us that, in a simply connected space with $\overline{\sec}_{\phi} \leq 0$, for  any $\widetilde{g}$-totally geodesic submanifold, $H$,  with $r_H$ the distance function to $H$, any modified distance function $u_H = h \circ r_H$ is smooth.  We can also show that if $|d\phi|\leq a$ then  there is a modified distance function to $H$ which is convex. 

\begin{lemma} Suppose that $(M,g)$ is a simply connected Riemannian manifold  admitting a density such that $\overline{\sec}_{\phi} \leq 0$ with $|d\phi|\leq a$ for some constant $a$.  Let $H$ be a totally geodesic submanifold in $\widetilde{g}$ metric, then there is a modified distance function  to $H$,  $u_H$, which is weighted convex.  Moreover, if $\overline{\sec}_{\phi} <0$ then $u_H$ is strictly weighted convex. 
\end{lemma} 

\begin{proof}
The proof is completely analogous to the proof of Theorem \ref{Thm:ModifiedDistanceConvex}. A similar second variation of energy argument  generalizes Lemma \ref{Lem:NonNegConvex} where the extra term at $t=0$ can be seen to vanish from $H$ being $\widetilde{g}$-totally geodesic.  We assume that $|d\phi| \leq a$  since, in repeating the proof of Theorem \ref{Thm:ModifiedDistanceConvex}, the function $a(t)$ may not exist in general since $H$ will not be compact unless it is a point. 
\end{proof}

With these preliminaries, we can establish Theorem \ref{Thm:Byers}.  The proof follows from a  similar  series of geometric and topological lemmas  as in the classical case, see \cite{Byers70} or \cite[Chapter 13, sec 2]{doCarmo76}.   In fact, there are only two parts of the argument that use curvature that we need to establish for the weighted curvatures:  that any covering transformation preserves at most one geodesic and that it is not possible for $\pi_1(M)$ to be cyclic if $M$ is compact.    

We fix some notation.  Consider $(M,g)$ to  be a compact manifold supporting a function $\phi$ with $\overline{\sec}_{\phi} < 0$. Let $\widehat{M}$ be the universal cover of $M$ with covering metric $\widehat{g}$ and let $\widehat{\phi}$ be the pullback of $\phi$ under the covering map. Let $F$ a covering transformation of $(\widehat{M},\widehat{g})$, since $\widehat{F}$ preserves $\phi$, $\widehat{F}$ is also an isometry of the conformal metric $\widetilde{g} = e^{-2\widehat{\phi}}\widehat{g}$.  An isometry $F$ of a Riemannian manifold is called a translation if it leaves invariant some geodesic which is called an axis of the translation.  For  the universal cover of a compact manifold, every covering transformation is a translation. Now we can prove the  two lemmas needed to prove Theorem \ref{Thm:Byers}.
 
\begin{lemma} Let $(M,g)$ be a compact manifold with NWSC.  Let $F$ be a covering transformation as above which is a translation along a $\widetilde{g}$-geodesic $\widetilde{\sigma}$.  Then $\widetilde{\sigma}$ is the unique $\widetilde{g}$-geodesic left invariant by $F$. 
\end{lemma}

\begin{proof}
Suppose that there are two $\widetilde{g}$-axes for $F$.  Call them $\widetilde{\sigma}_1$ and $\widetilde{\sigma}_2$.  Let $p \in \widetilde{\sigma}_2$.  Then there is a $g$-geodesic $\sigma$ which minimizes the distance from $p$ to $\widetilde{\sigma}_1$. Let $\alpha$ be the angle made by $\sigma$ and $\widetilde{\sigma}_2$ at $p$.  Consider the point $F(p)$.  Then $F\circ \alpha$ is a minimizing $g$-geodesic from $F(p)$ to $\widetilde{\sigma}_1$.  Moreover,  since $F$ is an isometry of both the $g$ and $\widetilde{g}$ metrics, the angles are preserved under $F$ and so the angle made by $F \circ \alpha$ and $\widetilde{\sigma}_2$ is also $\alpha$.

To see why this is a contradiction, consider the modified distance function to $\sigma_1$,  $u_{\widetilde{\sigma}_1}$ and consider its restriction to $\widetilde{\sigma}_2$.  Then it satisfies  $(e^{-2\phi} \widetilde{u}')' >0$ so that 
\begin{align*}
e^{-2\phi(\widetilde{\sigma_2}(t))} \widetilde{u}'(t) > e^{-2\phi(\widetilde{\sigma_2}(0))} \widetilde{u}'(0). 
\end{align*}
Note, however that $\phi(p) = \phi(F(p))$ and $u' = h' g(\nabla r, \widetilde{\sigma}_2') = h' e^{\phi} \alpha$.  Therefore,  the two sides of the equation must be equal at $p$ and $F(p)$, a contradiction. 
\end{proof}

\begin{lemma} Let $(M,g)$ be a compact manifold with NWSC then $\pi_1(M)$ is not infinite cyclic. \end{lemma}

\begin{proof} Suppose  $\pi_1(M)$ were  infinite cyclic. Then  all elements of $\pi_1(M)$ leave invariant a single $\widetilde{g}$-geodesic, $\widetilde{\sigma}$.  Let $\widehat{p} = \widetilde{\sigma}(0)$.  Let $\widehat{\beta}$ be a  unit speed $\widehat{g}$-geodesic with $\widehat{\beta}(0) = \widehat{p}$  that  is perpendicular to $\widetilde{\sigma}$. Let $p = \pi(\widehat{p})$ and consider the projection of $\widehat{\beta}$,  $\beta = \pi \circ \widehat{\beta}$ where $\pi$ is the covering projection.  

Since $M$ is compact, the geodesic $\beta$ must eventually stop being minimizing.  Consider a point $\beta(t_0)$ such that $\beta$ is not minimizing on $[0,t_0]$. Let $\alpha$ be a minimizing $g$-geodesic from $q=\beta(t_0)$ to $p$.  Let $\widehat{\alpha}$ be the lift of $\alpha$ starting from $\widehat{q}=\widehat{\beta}(t_0)$.  Since all elements of $\pi_1(M)$ leave $\widetilde{\sigma}$ invariant, the endpoint of $\widehat{\alpha}$ is on $\widetilde{\sigma}$.  

Consider $u=u_{\widehat q}$, the strictly convex modified distance function to $\widehat{q}$ restricted to the geodesic $\widetilde{\sigma}$.  Then, since $\widehat{\beta}$ and $\widetilde{\sigma}$ meet orthongonally, $\widetilde{u}'(0) = 0$.  By strict convexity, $\widetilde{u}(s)>\widetilde{u}(0)$  for all $s \neq 0$.  In particular,  this implies that   $\widehat{\alpha}$ has length at least $t_0$.  But this contradicts the choice of $t_0$. 

 \end{proof}

\begin{remark}
There are various other results for the fundamental group of compact manifolds with non-positive curvature.   From the work of Croke and  Schroeder \cite{CrokeSchroeder86}, Ivanov and Kapovitch \cite{IvanovKapovitch14} and others, most of these results have been generalized to metrics with out conjugate points.  Therefore, these results also hold for $\overline{\sec}_{\phi}\leq 0$. 
\end{remark}  

\subsection{Fixed point homogeneous spaces with positive curvature}

In this section we consider spaces with PWSC and symmetry.  In \cite{Wylie, KennardWylie17}, the first two authors prove that a number of classical results concerning manifolds with positive weighted sectional curvature generalize to the case of PWSC. Among these results are the classification of constant positive curvature, the Synge and Weinstein theorems, Berger's theorem on the vanishing of Killing fields, and Frankel's theorem and its generalization, Wilking's connectedness lemma.

In the presence of symmetry, \cite{KennardWylie17} contains further results and shows that much of the Grove symmetry program carries over to the case of PWSC. For example, for a compact Riemannian manifold admitting PWSC, the maximal rank of an isometric torus action is determined and shown to satisfy the same bound proved in Grove--Searle \cite{GroveSearle94} in the non-weighted setting (see \cite[Theorem C]{KennardWylie17}). In the equality case, called the case of maximal symmetry rank, Grove and Searle also prove a classification up to equivariant diffeomorphism. In \cite[Theorem C]{KennardWylie17}, the first two authors partially recover this statement up to homeomorphism. Here, we fully recover the classification of Grove and Searle in the weighted setting.

\begin{theorem}[Maximal symmetry rank]\label{thm:MSR}
Let $(M^n,g)$ be a closed Riemannian manifold that admits an effective action by a torus $T^r$. If $M$ has PWSC, then $r \leq \floor{\frac{n+1}{2}}$. Moreover, equality holds only if the action on M is equivariantly diffeomorphic to a linear action on $\s^n$, $\C\pp^{\frac{n}{2}}$, or a lens space.
\end{theorem}

The maximal symmetry rank classification of Grove and Searle, while significant on its own, has been applied in a large number of other classifications in the Grove Symmetry Program. For example, Wilking used not just the diffeomorphism, but the equivariant diffeomorphism, classification in his homotopy classification under the assumption of torus symmetry of roughly half-maximal rank. In \cite[Theorem D]{KennardWylie17}, the first two authors prove a weak version of Wilking's theorem that does not rely on Grove and Searle's equivariant classification. Equipped with Theorem \ref{thm:MSR}, together with the connectedness lemma and other results of \cite{KennardWylie17}, in the weighted setting, we are able to fully recover Wilking's classification (see \cite{Wilking03,DessaiWilking04}).

\begin{theorem}[Half-maximal symmetry rank]\label{thm:HalfMSR} 
Let $(M^n,g)$ be a closed, simply connected Riemannian manifold with $n \geq 11$ that admits an effective torus action of rank $r \geq \frac{n}{4} + 1$. If $M$ admits PWSC, then $M$ is tangentially homotopy equivalent to $\s^n$, $\C\pp^{\frac n 2}$, or $\HH\pp^{\frac n 4}$. In the case where $M$ is not simply connected, its fundamental group is cyclic.
\end{theorem}

Another application of Grove and Searle's equivariant diffeomorphism classification is due to Fang and Rong \cite[Corollary C]{FangRong05}. Again we fully recover this result in the weighted setting. 

\begin{theorem}[Almost maximal symmetry rank]\label{thm:AlmostMSR}
Let $(M^n,g)$ be a closed, simply connected Riemannian manifold of dimension $n \geq 8$ and symmetry rank $r \geq \frac{n}{2} - 1$. If $M$ admits PWSC, then it is homeomorphic to $\s^n$, $\C\pp^{\frac{n}{2}}$, or $\HH\pp^2$.
\end{theorem}

Equipped with Theorem \ref{thm:MSR}, as well as generalizations of results such as Berger's theorem and the connectedness lemma in the weighted setting proved in \cite{KennardWylie17}, the proofs of Theorems \ref{thm:HalfMSR} and \ref{thm:AlmostMSR} carry over without change are are omitted here. 

The proof of Theorem \ref{thm:MSR} also follows Grove and Searle's proof in the non-weighted case, but it requires some slight modifications and a new understanding of how positive curvature forces convexity in the weighted setting. The main difficulty is to recover the topological type of the manifold in the presence of an isometric circle action with fixed point set of codimension two. This situation is an example of what is called a fixed-point homogeneous action (defined below). Grove and Searle also classified such actions on manifolds with positive sectional curvature (see \cite{GroveSearle97}), and their result also generalizes to the case of PWSC:

\begin{reptheorem}{thm:FPH}
Let $(M,g)$ be a simply connected, closed Riemannian manifold that admits PWSC. If $M$ admits an isometric, fixed-point homogeneous action, then this action is equivariantly diffeomorphic to a linear action on a compact, rank one symmetric space.
\end{reptheorem}

The proofs of Theorems \ref{thm:MSR} and \ref{thm:FPH} are similar. For this reason, we only prove Theorem \ref{thm:FPH}, as it is more involved.

An isometric action of a connected Lie group $G$ on a Riemannian manifold $M$ is said to be fixed-point homogeneous if it is homogeneous or has the property that its fixed point set has a component $N$ such that the actions of $G$ on the unit normal spheres to $N$ are transitive. Equivalently, under the standard convention that the fixed point set $M^G$ has dimension $-1$ when it is empty, an action of $G$ on $M$ is fixed point homogeneous if and only if $\dim(M/G) = \dim(M^G)  + 1$. Note that in general, if $M^G$ is nonempty, then $M/G$ has dimension at least one more than $\dim(M^G)$, so fixed point homogeneity represents an extremal case.

A homogeneous Riemannian manifold $(M,g)$  with  PWSC has positive sectional curvature in the classical sense. This follows from Proposition \ref{Prop:LocallyHomogeneous} or by averaging $\phi$ as in \cite{KennardWylie17}. One immediately obtains a generalization to the weighted setting of the classifications in \cite{Wallach72,Berard-Bergery76,WilkingZiller} of homogeneous Riemannian manifolds with positive sectional curvature. We restrict attention here to the fixed-point homogeneous, but not homogeneous, case.

Throughout the proof, we consider the triple $(M, \tilde g = e^{2\phi} g, -\phi)$, and refer to geodesics with respect to $\tilde g$ as conformal geoesics. The key point where positive curvature plays a role is to prove the following:

\begin{lemma}\label{lem:tilde-convex}
Let $B_0 \subseteq M^G$ denote a component of the fixed point set that projects to a boundary component in $M/G$. For any horizontal, conformal geodesic $\tilde\sigma:[0,1] \to M$, the function $r \mapsto d(B_0, \tilde\sigma(r))$ does not achieve its minimum for any $r \in (0,1)$.\end{lemma}

\begin{proof}[Proof of Lemma \ref{lem:tilde-convex}]
Let $\tilde\sigma:[0,1] \to M$ be a horizontal, conformal geodesic, and assume some point in the interior of $\tilde\sigma$ achieves the minimum distance to $B_0$. Choose a horizontal geodesic $\sigma$ from $B_0$ to that point that realizes this distance. Note that $\sigma$ and $\tilde\sigma$ meet orthogonally by a first variation of energy argument.

We claim that there exists a vector field $V$ along $\sigma$ such that
	\begin{enumerate}
	\item $V$ is tangent to $B_0$ and $\tilde\sigma$ at the endpoints of $\sigma$,
	\item $V$ is orthogonal to the $G$--orbits along $\sigma$, and
	\item $V' = \nabla_{\sigma'} V$ is parallel to the $G$--orbits along $\sigma$.
	\end{enumerate}
Indeed, such a vector field exists as in the proof of \cite[Theorem 2.1]{Wilking03} since, by the fixed-point homogeneous assumption, the principal orbits have dimension $\delta \geq n - \dim(B_0) - 1$, which implies that the dimensions of $B_0$ and the image of $\tilde\sigma$ sum to at least $\dim(M/G)$.

Consider the variation $\sigma_r(t) = \widetilde{\exp}(r e^{\phi} V)$ of $\sigma$.  By a direct computation  using (\ref{eqn:IndexForm}) which is analogous to the argument in Lemma \ref{Lem:NonNegConvex}, the second variation satisfies
	\[\left.\frac{d^2}{dr^2}\right|_{r=0}E(\sigma_r)  = \int_a^b e^{2\phi} \of{  |V'|^2  -  g(R^{\nabla^{\phi}}(V,\sigma') \sigma', V)} dt. \]
Now consider Cheeger deformations $g_{\lambda}$ which shrink direction of the orbit. As was proven in \cite{KennardWylie17} the weighted sectional curvature only increases under the Cheeger deformation and, since $V'$ is parallel to the orbits, $|V'| _{g_{\lambda}}\rightarrow 0$ with $\lambda \rightarrow 0$.  Thus for some $\lambda$ small enough, $\left.\frac{d^2}{ds^2}\right|_{s=0}E(\sigma_r)<0$. This is a contradiction since $\sigma_r$ connects $B_0$ to $\tilde\sigma$ for all small $r$ and since $\sigma_0 = \sigma$ is also a minimum length path with respect to the metric $g_\lambda$.
\end{proof}

\begin{proof}[Proof of Theorem \ref{thm:FPH}]
Note by averaging that we may assume $\phi$ is $G$--invariant. Large parts of the proof in \cite{GroveSearle97} carry over to the case of PWSC. For example, the fact that $G$ acts transitively on the normal spheres places a strong restriction on $G$, namely, that is is one of the groups in \cite[(2.7)]{GroveSearle97}. As described below \cite[(2.7)]{GroveSearle97}, the classification follows from the Structure Theorem \cite[Theorem 2.2]{GroveSearle97} and the Uniqueness Lemma \cite[Lemma 2.5]{GroveSearle97}. Moreover, the Uniqueness Lemma is a differential topological statement in which curvature plays no role, so it also carries over to the present case. Hence, for our purposes, it suffices to show that the Structure Theorem carries over to the present case. 

The setup of the Structure Theorem is as follows (adopting notation from \cite{GroveSearle97}): $(M,g)$ is a compact Riemannian manifold that admits an almost effective, fixed point homogeneous, but not homogeneous, $G$--action. Let $B_0 \subseteq M^G$ denote a (non-empty) component of maximal dimension. The Structure Theorem states that all of the following hold under the assumption that $(M,g)$ has positive sectional curvature:
	\begin{enumerate}[label=(\roman*)]
	\item There is a unique ``soul orbit'' $B_1 = G\cdot p_1$ at maximal distance to $B_0$.
	\item All orbits in $M \setminus(B_0 \cup B_1)$ are principal and diffeomorphic to $\s^k \approx G/H$, the normal sphere to $B_0$, where $H$ is the principal isotropy group.
	\item There is a $G$--equivariant homeomorphism $M \approx DB_0 \cup_E DB_1$, where $DB_i$ denotes the normal disc bundle of $B_i$, and where $E$ denotes the common boundary of the $DB_i$ when viewed as tubular neighborhoods.
	\item All $G_{p_1}$--orbits in the normal sphere $\s^l$ to $B_1$ at $p_1$ are principal and diffeomorphic to $G_{p_1}/H$. Moreover, $B_0$ is diffeomorphic to $\s^l/G_{p_1}$.
	\end{enumerate}

We claim that each of these statements holds under the weaker assumption of $\overline\sec_\phi > 0$. First, (i) holds immediately by Lemma \ref{lem:tilde-convex}. To prove the remaining statements, we need to modify the proof from \cite{GroveSearle97}. The main change is that, instead of considering minimal geodesics $c_0$ and $c_1$ from $p$ to $B_0$ and from $p$ to $B_1$, respectively, we consider $g$--minimal geodesics $c_0$ and $\tilde g$--minimal geodesics $\tilde c_1$. The strategy then is exactly the same and the proof goes through with minor modifications. We proceed with the details.

To prove the remaining properties, we require the following angle condition, which is a slight refinement in this context of the one in \cite{GroveSearle97}:

	\begin{enumerate}
	\item[(v)] The angle between $c_0'(0)$ and $\tilde c_1'(0)$ is greater than $\frac{\pi}{2}$ for any minimal, horizontal geodesic $c_0$ from $p$ to $B_0$ and any minimal, horizontal, conformal geodesic $\tilde c_1$ from $p$ to $B_1$.
	\end{enumerate}

To prove this angle condition, let $p \in M \setminus (B_0 \cup B_1)$ and fix $c_0$ and $\tilde c_1$ as stated. By Lemma \ref{lem:tilde-convex}, the set $\{q \in M \st d(B_0, q) \geq d(B_0, p)\}$ is strictly convex with respect to $\tilde g$. In particular, the conformal geodesic $\tilde c_1$ from $p$ to $B_1$ has initial tangent vector pointing into this interior of this set (where the interior is defined in the sense of subsets of $M$ that are convex with respect to $\tilde g$). It follows that $\tilde c_1(s)$ lies in this set at least for all small $s > 0$. Suppose for a moment that the angle between $c_0'(0)$ and $\tilde c_1'(0)$ is less than $\frac \pi 2$. Choosing $\ep > 0$ appropriately small and replacing $c_0$ by a broken geodesic from $\tilde c_1(s)$ to $c_0(\ep)$ and then from $c_0(\ep)$ to $c_0(1) = p$, an argument using the first variation of energy formula implies that $\tilde c_1(s)$ is closer to $B_0$ than $p$, a contradiction. Similarly, if the angle between $c_0$ and $\tilde c_1$ is exactly $\frac \pi 2$, then one may apply the same argument to a small perturbation of $\tilde c_1$ given by a conformal geodesic starting at $p$ with initial vector given by $(\cos \theta) \tilde c_1'(0) + (\sin \theta) c_0'(0)$ for some sufficiently small $\theta > 0$. This again leads to a contradiction, so Property (v) follows.

We proceed to the proofs of Conditions (ii) -- (iv). For (ii) and (iv), one argues as in \cite{GroveSearle97}. To prove (iii), a bit more care is required. 

The strategy is to construct a vector field on $M$ satisfying the following properties:
	\begin{itemize}
	\item $Z$ is gradient-like for the distance function $d^g_{B_0} = d^g(B_0, \cdot)$ away from $B_0 \cup B_1$.
	\item $Z$ is radial near $B_0$ and $B_1$ (i.e., equal to $\nabla^{g} d^{g}_{B_0}$ on a neighborhood of $B_0$ and to $\nabla^{\tilde g} d^{\tilde g}_{B_1}$ near $B_1$).
	\end{itemize}
Given a vector field like this, we can construct a $G$--equivariant vector field that also satisfies these properties (since they are preserved under averaging along orbits of the group action). Hence it follows as in \cite{GroveSearle97} that $M$ is $G$--equivariantly homeomorphic to $DB_0 \cup _E DB_1$ as in the statement of Property (iii).

We construct the vector field $Z$ as follows. Fix $\epsilon > 0$ so that $B_0$ and $B_1$ have normal tubular $\ep$--neighborhoods 
	\begin{eqnarray*}
	B_0^\ep 	&=&	\{q \in M \st d^g(q,B_0) < \ep\},\\
	B_1^\ep 	&=&	\{q \in M \st d^{\tilde g}(q,B_1) < \ep\}.
	\end{eqnarray*}
On $M \setminus (B_0 \cup B_1^{\ep/3})$, let $X$ be a gradient-like vector field for $d^g_{B_0}$ that is radial on $B_0^{2\ep/3}$. This is possible on $M \setminus(B_0^{2\ep/3} \cup B_1^{\ep/3})$ by Condition (v), which implies that $d^g_{B_0}$ is regular there. In addition, $d^g_{B_0}$ is smooth on $B_0^\ep \setminus B_0$, so its gradient is defined and radial there. One uses a partition of unity to patch these definitions on the overlapping region $B_0^\ep \setminus B_0^{2\ep/3}$. By a similar construction, we obtain a vector field $Y$ on  $M \setminus (B_0^{\ep/3} \cup B_1)$ that is gradient-like for $d^{\tilde g}_{B_1}$ and is radial on $B_1^{2\ep/3}$. To construct a global vector field $Z$, note the following: If $p \in B_1^{\ep} \setminus B_1$, then $-Y = \tilde c_1'(0)$ for the minimal conformal geodesic $\tilde c_1$ from $p$ to $B_1$. Given any minimal geodesic $c_0$ from $p$ to $B_0$, the initial vector $c_0'(0)$ makes angle larger than $\frac \pi 2$ with $\tilde c_1'(0)$ by the angle condition above (Property (v)), so it makes angle larger than $\frac \pi 2$ with $-Y$. This shows that $-Y$ is also gradient-like for $d^g_{B_0}$ on $B_1^{2\ep/3} \setminus B_1$. Using a partition of unity,  construct a smooth vector field $Z$ satisfying the following properties:
	\begin{itemize}
	\item $Z = X$ on $M \setminus B_1^{2\ep/3}$. 
	\item $Z$ is a convex linear combination of $X$ and $-Y$ on $B_1^{2\ep/3} \setminus B_1^{\ep/3}$.
	\item $Z = -Y$ on $B_1^{\ep/3}$
	\end{itemize}
	
By the first and last conditions, $Z$ is radial near $B_0$ and $B_1$. Moreover, since $X$ and $-Y$ are gradient-like for $d^g_{B_0}$ on $M \setminus (B_0 \cup B_1^{\ep/3})$ and $B_1^\ep \setminus B_1$, respectively, $Z$ is gradient-like for $d^g_{B_0}$ on $M \setminus (B_0 \cup B_1)$. This completes the construction of a vector field $Z$ satisfying the two properties above, so the proof of Conditions (i) -- (iv), and hence of the theorem, is complete.
\end{proof}

For the case of fixed point homogeneous circle action, the normal spaces to the fixed point set must be two-dimensional. In other words, there is a submanifold of codimension two fixed by the circle action. This situation arises in the presence of a torus action of rank at least half the dimension of the manifold, so one immediately obtains diffeomorphism rigidity in the classification of maximal symmetry rank. In fact, the proof in \cite{GroveSearle94} also shows that one obtains equivariant rigidity for the entire torus action. Combining the Structure Theorem referenced in the proof of Theorem \ref{thm:FPH} with the arguments in \cite{GroveSearle94}, we recover the maximal symmetry rank classification of Grove and Searle for the case of PWSC.

We close this section with a discussion of isometric \textit{reflections} in the sense of Fang--Grove \cite{FangGrove16}. An isometric reflection is an isometry of order two that fixes a submanifold of codimension one. For a point in this submanifold, the normal sphere is zero-dimensional, i.e., a pair of points. Assuming the isometry acts non-trivially (equivalently, effectively), it acts transitively on this normal sphere. Hence the orbit space has boundary, and this may be viewed as a fixed point homogeneous action by $\Z_2$. Note that $\Z_2$ is the only finite group that can act effectively and fixed point homogeneously. In Fang--Grove \cite{FangGrove16}, the authors classify such actions on non-negatively curved manifolds. In the case of positive curvature, the proof is much simpler and only the sphere and real projective space arise. The argument in the positively curved case uses a similar strategy and again carries over to the case of PWSC. Hence we have the following:

\begin{corollary}[Reflections in PWSC]
Suppose a closed Riemannian manifold $(M,g)$ admits  PWSC. If $(M,g)$ admits an action by a reflection, then $M$ is diffeomorphic to $\s^n$ or $\R\pp^n$.
\end{corollary}
\section{Comparison estimates}

\subsection{Preliminaries}

Having established convexity results above for positive and negative weighted curvature, we now turn our attention to deriving optimal comparison estimates for non-zero curvature bounds.  While we have not directly used  $\nabla^{\phi}$ in the previous section, we must use the weighted connection for the more quantitative estimates in this section.  
Specifically, it turns out in our comparison estimates that the comparison functions must be parametrized in terms of the re-parametrization of geodesics coming from the connection $\nabla^{\phi}$.     We discuss this in  the next sub-section and apply it to the second variation formulas.  We also give   some simple examples showing that the use of the reparametrization is necessary for Jacobi field estimates.  

\subsubsection{Re-parametrization of geodesics and second variation of Energy}

Any connection gives rise to a notion of geodesics, which are the curves with zero acceleration.  We will call the geodesics for the connection $\nabla^{\phi}$ the $\phi$-geodesics and call the Riemannian geodesics $g$-geodesics.  Since $\nabla^{\phi}$ is projectively equivalent to the Levi-Civita connection, the $\phi$-geodesics are just reparametrizations of the  $g$-geodesics.  Given a $g$-geodesic, $\sigma(t)$, the parameter $s = \int e^{-2\phi( \sigma(t))} dt$ is the parameter of a $\phi$-geodesic.  We will say that a $\phi$-geodesic, $\gamma(s)$, is \emph{the standard re-parametrization of $\sigma(t)$} if the parameter $s$ is given by this formula in terms of $t$.   Below we will denote the $t$ derivative of a $g$-geodesic by $'$ and the $s$ derivative of the standard re-parametrization by $\dot{ }$.  We will use $\sigma$ for $g$-geodesics and $\gamma$ for $\phi$-geodesics. Any connection defines a notion of (geodesic) completeness, which is the condition that all geodesics can be extended for all time.  We say $(M,g, \phi)$ is $g$-complete is the Levi-Civita connection is complete and we say it is $\phi$-complete if $\nabla^{\phi}$ is complete. 

The \emph{re-parametrized distance} $s(p,q)$ is the globally defined function 
\begin{align} \label{eqn:sdistance}
  s(p,q) = \inf \left\{ s: \gamma(0) = p, \gamma(s) = q\right\},
\end{align}
where the infimum is taken over all normalized ${\phi}$-geodesics $\gamma$. That is $\phi$-geodesics which are the standard re-parametization of a minimizing unit speed $g$-geodesic. The function $s$ acts like the distance function in comparison estimates.

We apply this re-parametrization to the second variation formula.  Recall  equation (\ref{eqn:IndexForm}) for the index form which was proven in \cite{Wylie15}, 

	\begin{eqnarray*}
	\hspace{.1in}I(V,V)\hspace{-.1in} &=& \hspace{-.1in} \int_a^b \of{ |V' - d\phi(\sigma')V|^2 
	- R^{\nabla \phi}(V,\sigma',  \sigma', V)}dt
	+ \left.d\phi(\sigma') |V|^2\right|_{t=a}^{t=b}.	\end{eqnarray*}
	
This formula looks even closer to the standard formula for the second variation if we write it  in terms of $\phi$-geodesics. 

\begin{proposition}\label{Prop:NewIndex}
Given a manifold with density $(M,g,\phi)$ and a $\phi$-geodesic $\gamma:[a,b]\to M$ with standard parametrization in terms of $s$, and $V$ a vector-field along $\gamma$ everywhere orthogonal to $\dot\gamma$, then 
\[
I(e^\phi V,e^\phi V) = \int_a^b \left(\left|\nabla_{\dot\gamma} V\right|^2 - g(R^{\nabla^\phi}(V,\dot\gamma)\dot\gamma,V)\right)ds + d\phi(\dot\gamma)|V|^2\Big|_a^b.
\]
\end{proposition}

\begin{proof}
Using formula (\ref{eqn:IndexForm}) for  $\gamma:[a,b]\to M$ the standard re-parametrization of $\sigma$ we have 
\begin{align*}
I(e^\phi V, e^\phi V) &= \int_0^T \left(\left|e^\phi \nabla_{\sigma'} V + e^\phi d\phi(\sigma')V - e^\phi d\phi(\sigma')V\right|^2 - e^{2\phi} g(R^{\phi}(V,\sigma')\sigma',V)\right)dt\\
&\qquad\qquad\qquad+ e^{2\phi}d\phi(\sigma')|V|^2\Big|_{t=0}^{t=T}\\
&= \int_0^T e^{2\phi}\left(\left|\nabla_{\sigma'} V\right|^2 - g(R^{\nabla^\phi}(V,\sigma')\sigma',V)\right)dt + e^{2\phi}d\phi(\sigma')|V|^2\Big|_{t=0}^{t=T}\\
&= \int_0^T e^{-2\phi}\left(\left|\nabla_{\dot\gamma}V\right|^2 - g(R^{\nabla^\phi}(V,\dot\gamma)\dot\gamma,V)\right)dt + d\phi(\dot\gamma)|V|^2\Big|_{t=0}^{t=T}\\
&= \int_a^b \left(\left|\nabla_{\dot\gamma} V\right|^2 - g(R^{\nabla^\phi}(V,\dot\gamma)\dot\gamma,V)\right)ds + d\phi(\dot\gamma)|V|^2\Big|_{s=a}^{s=b}.
\end{align*}
\end{proof}

\begin{remark}  If we combine this formula with the arguments using the conformal change metric as in Lemma \ref{Lem:NonNegConvex} and Lemma \ref{lem:tilde-convex} we obtain the following second variation formula: Given a $\phi$-geodesic $\gamma:[a,b]\to M$ with standard parametrization in terms of $s$, and $V$ an orthogonal  vector-field along $\gamma$ ,
 then the variation $\gamma_r = \widetilde{\exp}(r e^{\phi} V)$ of $\gamma$ satisfies
	\[\left.\frac{d^2}{dr^2}\right|_{r=0}E(\gamma_r)  =  \int_a^b \left(\left|\nabla_{\dot\gamma} V\right|^2 - g(R^{\nabla^\phi}(V,\dot\gamma)\dot\gamma,V)\right)ds.	\]
\end{remark}

\begin{remark}  \label{Rem:CurveBound} The curvature term $g(R^{\nabla^\phi}(V,\dot\gamma)\dot\gamma,V)$ explains why is is natural to consider variable curvature bounds of the form $\overline{\sec}_{\phi} \geq \kappa e^{-4\phi}$ as we have $\dot{\gamma} = e^{2\phi} \sigma'$ so that the inequality  $g(R^{\nabla^\phi}(V,\dot\gamma)\dot\gamma,V) \geq \kappa$ holds for all standard reparametrizations of unit speed $g$-geodesics $\sigma$ and all $V$, unit perpendicular vector fields along $\gamma$, if and only if   $\overline{\sec}_{\phi} \geq \kappa e^{-4 \phi}$. \end{remark}

\subsubsection{Constant Radial curvatures} 
Now we note a fundamental difference between the usual sectional curvature and weighted sectional curvature.  Recall the result of Cartan-Ambrose-Hicks which states roughly that if we  have  two points $p$ and $q$ in two different Riemannian manifolds, with the property that all of the corresponding ``radial" sectional curvatures that involve planes containing geodesics emanating from points $p$ and $q$ are the same then the metrics are locally isometric.  In particular, if a point has constant sectional curvature for  all radial two planes, then a space has constant curvature.  

This result underlies many rigidity phenomena in comparison geometry as to show rigidity one shows that all radial curvatures are constant.  The following  examples show that this kind of rigidity is not true in the weighted case.

Let  $\phi$ be any function on the real line. Consider the metric 
\begin{align*}
g=dr^2 + e^{2\phi} \sn^2_k(s) g_{S^{n-1}} 
\end{align*}
where $s(r) = \int_0^r e^{-2\phi(t)} dt$  and  
\begin{align}
\label{Eqn:CompFunc} \sn_{\kappa}(s) = \left \{ \begin{array}{ccc} \frac{ \sin\left(\sqrt{\kappa}s\right) }{\sqrt{\kappa}} & \kappa>0 \\ s & \kappa= 0 \\ \frac{\sinh\left(\sqrt{-\kappa}s\right) }{\sqrt{-\kappa}} & \kappa<0 \end{array} \right. 
\end{align}
Recall that $\sn_{\kappa}$ is the solution to $\sn''_{\kappa} = - \kappa \sn_{\kappa}$, $\sn_{\kappa}(0) = 0$, $\sn'_{\kappa}(0)=1$. The motivation for defining $g$ is the following.  

\begin{proposition}
For the pair, $(g, \phi)$ as above if $X$ is a unit vector perpendicular to $\frac{\partial}{\partial r}$ then 
\begin{align*} \overline{\sec}_{\phi} \left( \frac{\partial}{\partial r}, X \right)  = k e^{-4\phi} \end{align*}
\end{proposition}

\begin{proof}
For a metric of the form $g_M = dr^2 + h^2(r,x)g_{S^{n-1}}$, 
\begin{align*}
\sec\left(\frac{\partial}{\partial r}, X \right) = -\frac{ \frac{\partial^2 h}{\partial r^2}}{h} 
\end{align*}

In our case,  $h = e^{\phi} \sn_k(s)$ where $s = \int e^{-2\phi} dr$.  So 
\begin{align*}
\frac{\partial h}{\partial r} &= \frac{\partial \phi}{\partial r} e^{\phi}  \sn_k(s) + e^{-\phi} \sn_k'(s), \\
\frac{\partial^2 h}{\partial r^2} &= e^{\phi}  \sn_k(s)\left( \frac{\partial^2 \phi}{\partial r^2} + \left(\frac{\partial \phi}{\partial r}\right)^2 \right) - ke^{-3\phi}\sn_k(s),
\end{align*}
from which the result follows. 

\end{proof}

The Jacobi fields of the metric $g$ are exactly $J(r) = e^{\phi}\sn_{\kappa}(s) E$ where $E$ is a perpendicular parallel field.  This shows that we can not expect uniform control on Jacobi fields that depends on the $g$-geodesic parametrization $r$.  We will prove an optimal Rauch comparison theorem  depending on the parameter $s$  in the next section. 

\subsection{Weighted Rauch theorems}

In this section we will prove the analogues of Rauch comparison theorems in the setting of manifolds with density. Recall that these theorems relate the growth rates of Jacobi fields on different manifolds, utilizing curvature bounds. Therefore, in order to prove analogues of the Rauch comparison theorems, we need to be able to compare the vector fields on two different manifolds with density. In particular, we need to be able to compare the index forms of two vector fields, provided that they satisfy certain conditions.

\begin{lemma}\label{Lem:Transfer}
Let $(M^n, g, \phi)$ and $(\widehat{M}^n, \widehat{g}, \widehat{\phi})$ be two manifolds with density. Let $\gamma,\widehat{\gamma}$ be geodesics with standard parametrization defined on $[0,S]$ on $M,\widehat{M}$ respectively.  Let $e_i,\widehat{e}_i$ be $g,\widehat{g}$-parallel, orthonormal basis along $\gamma,\widehat{\gamma}$ with $e_1\parallel\dot\gamma$ and  $\widehat{e}_1\parallel\dot{\widehat{\gamma}}$. Let
\[
V = \sum_{i=2}^n u_i(s)e_i(s)\qquad \widehat{V} = \sum_{i=2}^n u_i(s)\widehat{e}_i(s)
\]
and assume that $R^{\nabla^{\phi}}(V, \dot{\gamma}, \dot{\gamma}, V) \geq R^{\widehat{\nabla}^{\widehat{\phi}}}(\widehat{V}, \dot{\widehat{\gamma}},  \dot{\widehat{\gamma}}, \widehat{V})$ at each corresponding point $\gamma(s)$ and $\widehat{\gamma}(s)$, then
\[
I(e^\phi V, e^\phi V) \leq I(e^{\widehat\phi} \widehat V, e^{\widehat\phi}\widehat V) + \left[d\phi(\dot\gamma(s)) - d\widehat{\phi}(\dot{\widehat{\gamma}}(s))\right]\left|\widehat{V}(s)\right|^2\Big|_0^S.
\]
Moreover, equality occurs iff $R^{\nabla^{\phi}}(V, \dot{\gamma}, \dot{\gamma}, V) = R^{\widehat{\nabla}^{\widehat{\phi}}}(\widehat{V}, \dot{\widehat{\gamma}},  \dot{\widehat{\gamma}}, \widehat{V})$.
\end{lemma}

\begin{proof}
Note that 
\begin{align*}
\nabla_{\dot\gamma} V &= \nabla_{\dot\gamma} \left(\sum_{i=2}^n u_i(s)e_i(s)\right)  = \sum_{i=2}^n \dot u_i(s)e_i(s)\\
\widehat{\nabla}_{\dot{\widehat{\gamma}}} \widehat{V} &= \widehat{\nabla}_{\dot{\widehat{\gamma}}} \left(\sum_{i=2}^n u_i(s)\widehat{e}_i(s)\right)  = \sum_{i=2}^n \dot u_i(s)\widehat{e}_i(s).
\end{align*}
So that $|\nabla_{\dot\gamma} V|_g = |\widehat{\nabla}_{\dot{\widehat{\gamma}}} V|_{\widehat{g}}$.  We also clearly have $| V|_g = | \widehat{V}|_{\widehat{g}}$.
 
From Proposition~\ref{Prop:NewIndex}, we know that
\begin{align*}
  I(e^\phi V,e^\phi V) &= \int_0^S \left(\left|\nabla_{\dot\gamma} V\right|^2 - g(R^{\nabla^\phi}(V,\dot\gamma)\dot\gamma,V)\right)ds + d\phi(\dot\gamma)|V|^2\Big|_0^S\\
&\leq \int_0^S \left(\left|\widehat{\nabla}^{\widehat{\phi}}_{\dot{\widehat{\gamma}}}\widehat{V}\right|^2 - \widehat{g}(\widehat{R}^{\widehat{\nabla}^{\widehat\phi}}(\widehat{V},\dot{\widehat{\gamma}})\dot{\widehat{\gamma}},\widehat{V})\right)ds\\
&\qquad+ d\phi(\dot\gamma(S))|\widehat{V}(S)|^2 - d\phi(\dot\gamma(0))|\widehat{V}(0)|^2\\
&= I(e^{\widehat\phi} \widehat V, e^{\widehat\phi}\widehat V) + \left[d\phi(\dot\gamma(s)) - d\widehat{\phi}(\dot{\widehat{\gamma}}(s))\right]\left|\widehat{V}(s)\right|^2\Big|_0^S.
\end{align*}

The condition for equality follow immediately from the above comparison, since the one inequality corresponds precisely to the difference of curvatures.
\end{proof}

\begin{theorem}[First Rauch Comparison Theorem for Manifolds with Density]\label{Thm:Rauch1}
Let $(M^n, g, \phi)$ and $(\widehat{M}^n, \widehat{g}, \widehat{\phi})$ be two manifolds with density. Let $\gamma,\widehat{\gamma}$ be $\phi,\widehat{\phi}$-geodesics with standard parametrization defined on $[0,S]$ on $M,\widehat{M}$ respectively, with $\gamma$ having no conjugate points for $s\in[0,S]$.  Suppose that $R^{\nabla^{\phi}}(V, \dot{\gamma}, \dot{\gamma}, V) \geq R^{\widehat{\nabla}^{\widehat{\phi}}}(\widehat{V}, \dot{\widehat{\gamma}},  \dot{\widehat{\gamma}}, \widehat{V})$ for all unit vectors $V, \widehat{V}$ at the corresponding points $\gamma(s)$ and $\widehat{\gamma}(s)$

Let $J$ and $\widehat{J}$ be Jacobi fields along $\gamma,\widehat{\gamma}$ respectively. If
\[
J(0) = \widehat{J}(0) = 0\qquad |J'(0)|=|\widehat{J}'(0)|\qquad J'(0)\perp\dot\gamma(0)\qquad \widehat{J}'(0)\perp \dot{\widehat{\gamma}}(0),
\]
then
\[
e^{\phi(\gamma(0))-\phi(\gamma(s))}|J(s)| \leq e^{\widehat{\phi}(\widehat{\gamma}(0))-\widehat\phi(\widehat{\gamma}(s))}|\widehat{J}(s)|.
\]
\end{theorem}

\begin{proof}
Let $v(s) = |J(s)|^2, \widehat{v}(s)=|\widehat{J}(s)|^2$. For an arbitrary $s_0\in[0,S]$, define two new Jacobi fields:
\[
U(s) = \frac{1}{\sqrt{v(s_0)}}J(s)\qquad \widehat{U}(s) = \frac{1}{\sqrt{\widehat{v}(s_0)}}\widehat{J}(s).
\]
Observe that
\begin{align*}
\frac{\dot{v}(s_0)}{v(s_0)} &=e^{2\phi(\gamma(s_0))}\frac{v'(s_0)}{v(s_0)} = 2e^{2\phi(\gamma(s_0))} I(U,U). 
\end{align*}
Similarly,
\[
\frac{\dot{\widehat{v}}(s_0)}{\widehat{v}(s_0)} = 2e^{2\widehat{\phi}(\widehat{\gamma}(s_0))} I(\widehat{U},\widehat{U}).
\]

Choose parallel orthonormal bases along $\gamma,\widehat{\gamma}$ such that $U(s_0) = e_2(s_0)$ and $\widehat{U}(s_0) = \widehat{e}_2(s_0)$. We now apply Lemma~\ref{Lem:Transfer} with $\widehat{V}(s) = e^{\widehat{\phi}(\widehat{\gamma}(s_0))-\widehat{\phi}(\widehat{\gamma}(s))}\widehat{U}(s)$ and $V$ the corresponding field along $\gamma$, so
\[
I(e^\phi V, e^\phi V) \leq I(e^{\widehat\phi}\widehat{V},e^{\widehat\phi}\widehat{V}) + \left[d\phi(\dot\gamma(s)) - d\widehat\phi(\dot{\widehat{\gamma}}(s))\right]|\widehat{V}(s)|^2\Big|_0^{s_0}.
\]
However, we know that $\widehat{V}(0) = 0$ and $|\widehat{V}(s_0)| = 1$, so we get
\[
I(e^\phi V, e^\phi V) \leq e^{2\widehat{\phi}(\widehat{\gamma}(s_0))}I(\widehat{U},\widehat{U}) + \left[d\phi(\dot\gamma(s_0)) - d\widehat\phi(\dot{\widehat{\gamma}}(s_0))\right].
\]

We now consider another vector field along $\gamma$ defined by $W = e^{-\phi(\gamma(s_0))}V$, then $I(e^\phi W, e^\phi W) = e^{-2\phi(\gamma(s_0))} I(e^\phi V, e^\phi V)$ and $(e^\phi W)(s_0) = U(s_0)$, so
\begin{equation}\label{ineq:Rauch1}
e^{2\phi(\gamma(s_0))} I(U,U) \leq I(e^\phi V,e^\phi V) \leq e^{2\widehat{\phi}(\widehat{\gamma}(s_0))}I(\widehat{U},\widehat{U}) + \left[d\phi(\dot\gamma(s_0)) - d\widehat\phi(\dot{\widehat{\gamma}}(s_0))\right],
\end{equation}
which we re-write as
\begin{align}
\frac{\dot{v}(s_0)}{v(s_0)} - 2d\phi(\dot\gamma(s_0)) \leq \frac{\dot{\widehat{v}}(s_0)}{\widehat{v}(s_0)} - 2d\widehat\phi(\dot{\widehat{\gamma}}(s_0)). \label{Eqn:RauchDerivative}
\end{align}
Since $s_0$ was arbitrary, we can solve this differential inequality as follows:
\[
e^{2\phi(\gamma(0))-2\phi(\gamma(s))}|J(s)|^2 \leq e^{2\widehat{\phi}(\widehat{\gamma}(0))-2\widehat\phi(\widehat{\gamma}(s))}|\widehat{J}(s)|^2.
\]
as was claimed.
\end{proof}

We now move on to the second Rauch Comparison Theorem, also called Berger Comparison Theorem.

\begin{theorem}[Second Rauch Comparison Theorem for Manifolds with Density]\label{Thm:Rauch2}
Let $(M^n,g,\phi)$ and $(\widehat{M}^n,\widehat{g},\widehat{\phi})$ be manifolds with density. Let $\gamma:[0,S]\to M$ and $\widehat{\gamma}:[0,S]\to\widehat{M}$ be $\phi$-geodesics with standard parametrization, and $\gamma$ having no focal points to the geodesic submanifold given by $\exp_{\gamma(0)} \dot\gamma(0)^\perp$. Suppose that $R^{\nabla^{\phi}}(V, \dot{\gamma}, \dot{\gamma}, V) \geq R^{\widehat{\nabla}^{\widehat{\phi}}}(\widehat{V}, \dot{\widehat{\gamma}},  \dot{\widehat{\gamma}}, \widehat{V})$ for all unit vectors $V, \widehat{V}$ at the corresponding points $\gamma(s)$ and $\widehat{\gamma}(s)$. Furthermore, let $J$ and $\widehat{J}$ be Jacobi fields along $\gamma,\widehat{\gamma}$ respectively, parametrized in terms of $s$. If
\[
J'(0) = \widehat{J}'(0) = 0 \qquad |J(0)| = |\widehat{J}(0)| \qquad J(0)\perp\dot\gamma(0)\qquad \widehat{J}(0)\perp\dot{\widehat{\gamma}}(0),
\]
then
\[
e^{\phi(\gamma(0))-\phi(\gamma(s))}|J(s)|\leq e^{\widehat\phi(\widehat{\gamma}(0))-\widehat\phi(\widehat{\gamma}(s))}|\widehat{J}(s)| e^{(d\widehat{\phi}(\dot{\widehat{\gamma}}(0))-d\phi(\dot\gamma(0)))\tau(s)},
\]
where
\[
\tau(s) = \int_0^s \frac{|e^{-\widehat\phi}\widehat{J}|^2(0)}{|e^{-\widehat\phi}\widehat{J}|^2(\xi)}d\xi.
\]
\end{theorem}

\begin{remark}
In the special case where $\widehat\phi=0$ and $\widehat{M}$ has $\sec\equiv K$, $\tau$ is a generalized tangent:
\[
\tau(s) = \begin{cases}
s & K = 0\\
\frac{1}{\sqrt{K}}\tan(\sqrt{K}s) & K>0\\
\frac{1}{\sqrt{-K}}\tanh(\sqrt{-K}s) & K<0
\end{cases}
\]
\end{remark}

\begin{proof}
Define $v,\widehat{v},U,\widehat{U},V,\widehat{V}$ as in the Proof of Theorem~\ref{Thm:Rauch1}, then we get:
\[
I(e^\phi V, e^\phi V) \leq I(e^{\widehat\phi}\widehat{V},e^{\widehat\phi}\widehat{V}) + [d\phi(\dot\gamma(s)) - d\widehat{\phi}(\dot{\widehat{\gamma}}(s))]|\widehat{V}(s)|^2\Big|_0^{s_0},
\]
We still have $|\widehat{V}(s_0)| = 1$, however, this time we have $|\widehat{V}(0)| = e^{\widehat\phi(\widehat{\gamma}(s_0))-\widehat\phi(\widehat{\gamma}(0))}|\widehat{U}(0)| = e^{\widehat\phi(\widehat{\gamma}(s_0))-\widehat\phi(\widehat{\gamma}(0))} \frac{|\widehat{J}(0)|}{|\widehat{J}(s_0)|} = \frac{|e^{-\widehat\phi}\widehat{J}|(0)}{|e^{-\widehat\phi}\widehat{J}|(s_0)}$. As before, define $W= e^{-\phi(\gamma(s_0))}V$, then using the Index Lemma, we get:
\begin{align}
\begin{split}\label{ineq:Rauch2}
e^{2\phi(\gamma(s_0))}I(U,U) &\leq I(e^\phi V, e^\phi V) \\
 &\leq e^{2\widehat{\phi}(\widehat{\gamma}(s_0))} I(\widehat{U},\widehat{U}) + [d\phi(\dot\gamma(s_0))-d\widehat\phi(\dot{\widehat{\gamma}}(s_0))] + [d\widehat\phi(\dot{\widehat{\gamma}}(0)) - d\phi(\dot\gamma(0))]\left[\frac{|e^{-\widehat\phi}\widehat{J}|(0)}{|e^{-\widehat\phi}\widehat{J}|(s_0)}\right]^2,
 \end{split}
\end{align}
which can be re-written as:
\[
\frac{\dot{v}(s_0)}{v(s_0)} - 2d\phi(\dot\gamma(s_0)) \leq \frac{\dot{\widehat{v}}(s_0)}{\widehat{v}(s_0)} - 2d\widehat\phi(\dot{\widehat{\gamma}}(s_0)) + [d\widehat\phi(\dot{\widehat{\gamma}}(0)) - d\phi(\dot\gamma(0))]\left[\frac{|e^{-\widehat\phi}\widehat{J}|(0)}{|e^{-\widehat\phi}\widehat{J}|(s_0)}\right]^2
\]
Since $s_0$ was arbitrary, we can solve this and obtain the claimed result.
\end{proof}

\begin{proposition}
Equality in Theorem~\ref{Thm:Rauch1} and Theorem~\ref{Thm:Rauch2} occurs when the following conditions are satisfied:

\begin{enumerate}
\item\label{Cond:par} $J(s) = |J(s)|e_2(s)$, $\widehat{J}(s) = |\widehat{J}(s)|\widehat{e}_2(s)$, where $e_2,\widehat{e}_2$ are $g,\widehat{g}$-parallel unit vectors orthogonal to $\dot{\gamma},\dot{\widehat{\gamma}}$ respectively.
		
\item\label{Cond:curv} $R^{\nabla^\phi}(e_2(s),\dot{\gamma},\dot{\gamma},e_2(s)) = \widehat{R}^{{\widehat{\nabla}}^{\widehat{\phi}}}(\widehat{e}_2(s),\dot{\widehat{\gamma}},\dot{\widehat{\gamma}},\widehat{e}_2(s)$ for all $s\in[0,T]$.
\end{enumerate}
\end{proposition}

\begin{proof}
In the proof of Theorem~\ref{Thm:Rauch1}, the only inequalities we had were the two in \eqref{ineq:Rauch1}. In the proof of Theorem~\ref{Thm:Rauch2}, the only inequalities were the two in \eqref{ineq:Rauch2}
	
The inequality on the left in \eqref{ineq:Rauch1} and \eqref{ineq:Rauch2} arises from the Index Lemma, and leads to condition \eqref{Cond:par} above. Equality in Index Lemma occurs precisely when the vector field in question equals the Jacobi field, so we can conclude that $e^\phi V = e^{\phi(\gamma(s_0))} U$. Let $U(s) = \sum_{i=2}^n u_i(s)e_i(s)$ and $\widehat{U}(s) = \sum_{i=2}^n \widehat{u}_i(s)\widehat{e}_i(s)$, where $e_i,\widehat{e}_i$ are $g$ and $\hat{g}$ parallel orthonormal basis along $\gamma$ and $\widehat{\gamma}$ respectively. By construction,
\begin{align*}
\widehat{V}(s) &= \sum_{i=2}^n e^{\widehat{\phi}(\widehat{\gamma}(s_0)) - \widehat{\phi}(\widehat{\gamma}(s))}\widehat{u}_i(s) \widehat{e}_i(s)\\
V(s) &= \sum_{i=2}^n e^{\widehat{\phi}(\widehat{\gamma}(s_0)) - \widehat{\phi}(\widehat{\gamma}(s))}\widehat{u}_i(s) e_i(s).
\end{align*}
Therefore,
\[
U(s) = \sum_{i=2}^n e^{\widehat{\phi}(\widehat{\gamma}(s_0)) - \widehat{\phi}(\widehat{\gamma}(s)) - \phi(\gamma(s))}\widehat{u}_i(s) e_i(s).
\]
However, the choice of $e_i,\widehat{e}_i$ other than $i=2$ was completely arbitrary and independent of each other. Therefore, the only way this can happen is if $\widehat{u}_i(s) = 0$ for $i\neq 2$. Therefore, $U,\widehat{U}$ are $g,\widehat{g}$-parallel up to scaling, and so are $J,\widehat{J}$ as claimed.
	
The inequality on right of \eqref{ineq:Rauch1} and \eqref{ineq:Rauch2} arises from Lemma~\ref{Lem:Transfer} and leads to condition \eqref{Cond:curv} above by the equality case of Lemma~\ref{Lem:Transfer}.
\end{proof}

\subsection{The Sphere Theorem}

As an application of the Rauch comparison theorem we will prove the sphere theorem mentioned in the introduction.  In fact, using the Rauch comparison theorems we get the same conjugate and injectivity radius estimates  as are used in the classical case.  

  For submanifolds $A$ and $B$ in $M$ define the path space as  
 \[ \Omega_{A,B}(M) = \{ \gamma:[0,1] \rightarrow M, \gamma(0) = A, \gamma(1) = B \}
 \]  
 We consider the Energy $E:  \Omega_{A,B}(M) \rightarrow \mathbb{R}$ and  variation fields tangent to $A$ and $B$ at the end points. The  critical points are then  the geodesics perpendicular to $A$ and $B$ and  we say that the index of such a geodesic is $\geq k$ if there is a $k$-dimensional space of variation fields along the geodesic which have negative second variation.  The first result is the following. We state the results in this subsection in terms of the the invariants $\underline{\kappa}_g$ and $\overline{K}_g$ defined in Definition \ref{Def:CurveBound}. 

\begin{lemma} \label{Lem:index} Suppose that $(M,g)$ is a Riemannian manifold such that $\underline{\kappa}_g >0$.  Let $\sigma$ be a unit speed geodesic of length $> \frac{\pi}{\sqrt{\underline{\kappa}_g}}$  then the index of $\sigma$ is greater than or equal to $(n-1)$.
\end{lemma}

\begin{proof}
From the definition of $\underline{\kappa}_g$,  for each $\varepsilon>0$, we have densities  $\phi_{\varepsilon}\leq 0$ such that  $\overline{\sec}_{\phi_\varepsilon} \geq (\underline{\kappa}_g - \varepsilon) e^{-4\phi_{\varepsilon}}$. Let $\gamma_{\varepsilon}(s_{\varepsilon})$ be the standard reparametrization of $\sigma(r)$ with respect to $\phi_{\varepsilon}$.  Since $\phi_\varepsilon \leq 0$, $s_{\varepsilon} \geq r$, so for $\varepsilon$ sufficiently small, $\gamma_{\varepsilon}$ is defined for $s_{\varepsilon} \in [0, T_0]$, for some $T_0> \frac{\pi}{\sqrt{\underline{\kappa}_g}}$. 

Take $\varepsilon$ sufficiently small that  so that $T_0> \frac{\pi}{\sqrt{\underline{\kappa}_g-\varepsilon}}$. Apply Theorem \ref{Thm:Rauch1} to the geodesic $\gamma_{\varepsilon}$ in the space with density $(M,g, \phi_{\varepsilon})$ and  $\widehat{M}$  the sphere with metric of  constant curvature $\underline{\kappa}_g - \varepsilon$ and $\widehat{\phi_\varepsilon} \equiv 0$. Then $\widehat{J}(s) = \sin(\sqrt{\underline{\kappa}_g - \varepsilon} s)$.  Since $\widehat{J}$ has a zero at $\frac{\pi}{\sqrt{\underline{\kappa}_g-\varepsilon}}$, every orthogonal Jacobi field to $\gamma_{\varepsilon}$ must have a zero in the interval $[0, \frac{\pi}{\sqrt{\underline{\kappa}_g-\varepsilon}}]$.  Since $\sigma$ is just a re-parametrization of $\gamma_{\varepsilon}$ this implies $\sigma$ must have  index $\geq n-1$. 
\end{proof}

\begin{remark} Lemma \ref{Lem:index} implies Theorem \ref{Thm:Myers}. \end{remark}

We can also obtain a lower bound on the conjugate radius from an upper bound on $\overline{K}_g$.

\begin{lemma} \label{Lemma:ConjPos} Suppose that $(M,g)$ is a Riemmanian manifold with $\overline{K}_g >0$.  Let $\sigma$ be a unit speed $g$-geodesic, then any two conjugate points of $\sigma$ are distance $\geq \frac{\pi}{\sqrt{\overline{K}_g}}$ apart. 
\end{lemma}

\begin{proof}
Let $\gamma_{\varepsilon}(s)$ be the standard reparametrization of $\sigma(r)$ with respect to densities $\phi_{\varepsilon}$ with $\overline{\sec}_{\phi_{\varepsilon}} \leq (\overline{K}_g + \varepsilon)e^{-4\phi_{\varepsilon}}$ and $\phi_{\varepsilon} \geq 0$, $s \leq r$.  Apply Theorem \ref{Thm:Rauch1} in the opposite way than in the previous lemma. 
\end{proof}

Now that we have Lemmas \ref{Lem:index} and \ref{Lemma:ConjPos} we have the same control on the index of long geodesics and the conjugate radius as one has for un-weighted curvature bounds.  These facts, along with the resolution of the Poincare conjecture, allow us to prove the quarter-pinched sphere theorem using a classical argument of Berger.  The key observation is that if all geodesics in $\Omega_{p,p}$ have index at least $(n-1)$, then $\Omega_{p,p}$ is $(n-2)$-connected and hence $M$ is $(n-1)$-connected.  If $M$ is compact, this implies the manifold is a homotopy sphere, and thus homeomorphic to the sphere by the resolution of the Poincare conjecture, see \cite[Theorem 6.5.3]{Petersen16} for details.      
 
In order to use this result we must prove the injectivity radius estimate. These now follow with the classical proofs which we sketch for  completeness. 

\begin{lemma} \label{Lemma:EvenInjectivity} Suppose that $(M^n, g)$ is an even dimensional orientable manifold with $\underline{\kappa}_g >0$ and $\overline{K}_g \leq 1$  then $\mathrm{inj}_M \geq \pi$. 
\end{lemma}

\begin{proof}
Since $\overline{K} \leq 1$, for all $\varepsilon>0$ there are densities $\phi_{\varepsilon}$ such that $\overline{\sec}_{\phi_{\varepsilon}} \leq (1+\varepsilon)e^{-4\phi_{\varepsilon}}$.  Suppose that $\mathrm{inj}_M  < \pi$ and let $p,q$ be two points such that $d(p,q) =\mathrm{inj}_M$. Let $\sigma$ be a unit speed minimizing geodesic from $p$ to $q$.   Let $s_{\varepsilon}$ be the standard re-parametrization parameter of $\sigma$ with respect to the density $\phi_{\varepsilon}$.  Since $\phi_{\varepsilon} \geq 0$, we have  $s_{\varepsilon} < d(p,q) < \pi$, so for $\varepsilon$ small enough we have $s_{\varepsilon} < d(p,q) < \frac{\pi}{\sqrt{1+\varepsilon}}$.   By Lemma \ref{Lemma:ConjPos} $\sigma$ does not have conjugate points.  Then, using a standard argument, there must be a closed geodesic through $p$ and $q$ which is the shortest closed geodesic in $M$.  But this is impossible with $\underline{\kappa}>0$ by the Synge argument, which shows under the dimension and orientability hypothesis that any closed geodesic can be homotoped to have smaller length, see \cite[Theorem 5.4]{Wylie15}.
\end{proof}

In the odd dimensional case we have the following injectivity radius estimate. 

\begin{lemma} \label{Lemma:Inj} Suppose that $(M^n,g)$ is a simply connected manifold with $\underline{\kappa}_g >\frac{1}{4}$ and $\overline{K}_g \leq 1$ then $\mathrm{inj}(M,g) \geq \pi$. \end{lemma}

\begin{proof}  From Lemma \ref{Lemma:ConjPos} we know that every geodesic of length $\leq \pi$ does not have conjugate points. While from Lemma \ref{Lem:index} we know that there is a positive constant $\delta$ such that  every geodesic of length $\geq 2\pi - \delta$ has index $\geq 2$.  These are the only two facts about curvature used in Klingerberg's original proof of the injectivity radius estimate, so his proof goes through.  See, for example, \cite[Theorem 6.5.5]{Petersen16}.
\end{proof}
This now gives us the sphere theorem. 
 
 \begin{theorem} 
 Let $(M,g)$ be a simply connected smooth complete manifold with $\underline{\kappa}_g>0$ and $\delta > \frac{1}{4}$, then $M$ is homeomorphic to the sphere. 
\end{theorem} 

\begin{proof}
 By rescaling the metric (but not the density) we can assume that $\overline{K}= 1$ and $\underline{\kappa} > \frac{1}{4}$.  By Lemma \ref{Lemma:Inj} we have $\mathrm{inj}_M \geq \pi$.  Thus any geodesic loop, $\sigma$,  in $\Omega_{p,p}$ must that length $\geq 2\pi$.  Let $\phi$ be a density with $\overline{\sec}_{\phi} > \frac{1}{4} e^{-4\phi}$ and $\phi \leq 0$.  Then $s = \int_0^t e^{-2\phi(\sigma(t))} dt \geq \mathrm{length}(\sigma) \geq 2\pi$.  By Lemma \ref{Lem:index}  any closed geodesic must then have index at least $(n-1)$. 
   \end{proof}

\subsection{Hessian comparison theorem} 
In this section we link the weighted Rauch comparison theorem to the Hessian of the distance function in the conformal metric as was discussed in Section 2. Consider a point $p$ and $r_p(x) = r(x) = d(p,x)$ the $g$-distance to $p$.  Let $q$ be a point so that $r_p$ is smooth in a neighborhood of $q$, and let $Y \in T_q M$  with $Y \perp \nabla r$.  In Section 2 we considered the quantity 
\begin{align*}
\left( \mathrm{Hess}_g r - d\phi(\nabla r) g \right)(Y,Y) =  \left( \mathrm{Hess}_{\widetilde{g}} r - d\phi\otimes dr - dr \otimes d\phi \right)(Y,Y)  
\end{align*}
where $\widetilde{g} = e^{-2\phi} g$.  

Recall that for any Jacobi field $J$ which is perpendicular to $\nabla r$ at a point $x$ where $r$ is smooth  it follows from the second variation of energy formula that
\begin{align*}
g(J', J) = \mathrm{Hess} r(J,J)
\end{align*}
where $J'$ is the derivative of $J$ along the unique unit speed geodesic from $p$ to $x$. Then if we consider the Jacobi field $J(s)/|J(s_0)|$ we have 
\begin{align}
e^{2\phi} \left( \mathrm{Hess} r \left(\frac{J(s)}{|J(s_0)|},\frac{J(s)}{|J(s_0)|} \right) - d\phi(\nabla r) \right) 
&=e^{2\phi} \left( \frac{ g(J', J)}{|J(s_0)|^2} - d\phi(\nabla r) \right) \nonumber \\
&=\frac{1}{2}\left( \frac{\dot{v}(s_0)}{v (s_0)} - 2 d\phi\left( \frac{d}{ds} \right) \right) \label{Eqn:JacobiHessian}
\end{align}
where $v(s) = |J(s)|^2$. This is exactly the quantity that we bounded in the proof of Theorem \ref{Thm:Rauch1}.

Putting this together gives us the following comparison.  Recall that $\sn_{\kappa}$ is the standard comparison function as defined in equation (\ref{Eqn:CompFunc}).  Let  $\cs_{\kappa} = \sn_{\kappa}'$ and  recall that the Hessian of the distance function in a simply connected space of constant curvature $\kappa$ is given by $\frac{\cs_{\kappa}}{\sn_{\kappa}}.$

\begin{theorem}[Hessian Comparison]\label{Thm:HessianComp}   Suppose that $(M,g, \phi)$ is a Riemannian manifold with density. Fix a point $p$ and let $r$ be the distance to $p$.  Let $q$ be a point such that the distance function to $p$ is smooth at $q$ and let $Y \in T_q M$ be a unit length vector such that $Y \perp \nabla r$.  \begin{enumerate}
\item  If, for all unit vectors $Z$ perpendicular to the minimizing geodesic from $p$ to $q$, 
$\overline{\sec}_{\phi}(Z, \nabla r)\geq  \kappa e^{-4\phi}$, 
then 
\begin{align*}
\left( \mathrm{Hess}_g r - d\phi(\nabla r) g \right)(Y,Y) \leq e^{-2\phi(q)} \frac{\cs_{\kappa} (s(p,q))}{\sn_{\kappa}(s(p,q))}.
\end{align*}
\item  If, for all unit vectors $Z$ perpendicular to the minimizing geodesic from $p$ to $q$, 
 $\overline{\sec}_{\phi}(Z, \nabla r) \leq Ke^{-4\phi}$, then 
\begin{align*}
\left( \mathrm{Hess}_g r - d\phi(\nabla r) g \right)(Y,Y)  \geq e^{-2\phi(q)} \frac{\cs_{K}(s(p,q))}{\sn_K(s(p,q))}.
\end{align*}
\end{enumerate}
Where $s(p,q)$ is the reparametrized distance defined in (\ref{eqn:sdistance}). 

\end{theorem}

\begin{proof}
We outline the proof of the first inequality.  The second is completely analogous. Consider Theorem \ref{Thm:Rauch1} with $(M,g, \phi)$ our given manifold with density and $(\widehat{M}^n, \widehat{g}, \widehat{\phi})$ the standard model space of constant curvature $\kappa$ and $\widehat{\phi} \equiv 0$.  Let $J$ be the unique Jacobi field with $J(0) = 0$ and $J(s_0)= Y$ where $s_0 = s(p,q)$.  Let $\widehat{J} = |J'(0)| \sn_{\kappa}(s)E$ be the corresponding Jacobi field in $\widehat{M}$.   Then letting $\widehat{v} = |\widehat{J}|^2 = |J'(0)|^2 |\sn_{\kappa}(s)|^2$ and combining (\ref{Eqn:JacobiHessian}) and (\ref{Eqn:RauchDerivative})  we obtain
\begin{align*}
e^{2\phi(q)} \left( \mathrm{Hess}_g r - d\phi(\nabla r) g \right)(Y,Y)  &= \frac{e^{2\phi}}{2}\left( \frac{\dot{v}(s_0)}{v (s_0)} - 2 d\phi\left( \frac{d}{ds} \right) \right) \\
&\leq  \frac12 \frac{\dot{\widehat{v}}(s_0)}{\widehat{v}(s_0)} - d\widehat\phi(\dot{\widehat{\gamma}}(s_0)) \\
&= \frac{\cs_{\kappa}(s(p,q))}{\sn_{\kappa}(s(p,q))}
\end{align*}
\end{proof}

\begin{remark} Theorem 5.17 is optimal in general.  However, note that the function 
\[ t \rightarrow \frac{\cs_{\kappa}(t)}{\sn_{\kappa}(t)} \]
is monotonically decreasing in $t$.   Thus,  if $\phi \leq 0$, then $s(p,q) \geq  r_p(q)$ and $\frac{\cs_{\kappa} (s(p,q))}{\sn_{\kappa}(s(p,q))} \leq \frac{\cs_{\kappa} (r(q))}{\sn_{\kappa}(r(q))}$.  Thus if we assume $\underline{\kappa}_g \geq \kappa>0$ we can replace the $s(p,q)$ on the right hand side of (1) with the distance $r(q)$.  All these inequalities are reversed if we assume $\overline{K}_g \leq K$, for some $K>0$, so we can similarly replace $s(p,q)$ with $r(q)$ in (2) in this case. 
\end{remark}

We can now prove the Cheeger finiteness theorem for positive curvature and even dimensions. 

\begin{reptheorem}{IntroThm:FinitePosEven}
For given $n, a>0$ and $0<\delta_0 \leq 1$ the class of Riemannian $2n$-dimensional manifolds with $\underline{\kappa}(a) >0$ and $\delta(a) \geq \delta_0$ contain  only finitely many diffeomorphism types. 
\end{reptheorem} 

\begin{proof}

As is standard in convergence theory, we can show there are only finitely many diffeomorphism types by showing the class is compact in the $C^{\alpha}$ topology.  Moreover,  such compactness is true is there is a uniform upper bound on diameter, lower bound on injectivity radius, and two-sided bound on the Hessian of the distance function inside balls of a uniform fixed radius. See, for example, \cite{Petersen97} for a survey.   

Lemma \ref{Lem:index} gives the upper bound on diameter and  Lemma \ref{Lemma:EvenInjectivity} gives a lower bound on injectivity radius. Once we have the upper bound on diameter since, by Remark \ref{Rem:Zero}, we can choose $\phi$ so that there is a point where $\phi(p) = 0$, the assumption $|d\phi| \leq a$ implies  there is a constant $B$, depending on $a$ and the diameter bound, such that $|\phi| \leq B$.  Then Theorem \ref{Thm:HessianComp} provides the required two sided bounds on the Hessian of the distance function.  
\end{proof}

In order to prove the more general finiteness theorem, Theorem \ref{IntroThm:GenFinite}, the only further ingredient we require is a lower bound on the length of a closed geodesic that depends on a two-sided bound on $\phi$,  a lower bound on $\overline{\sec}_{\phi}$, an upper bound on diameter, and lower bound on volume.  We establish such an estimate in the next section.  In fact, the bound will follow from a more general set of formulas for volumes of tubes around submanifolds of arbitrary codimension in a manifold with weighted sectional curvature lower bounds.  In the non-weighted setting these estimates are due to Heintze and Karcher \cite{HeintzeKarcher78}.

\subsection{Tube volumes}

In this section we prove the weighted Heintze-Karcher theorem \cite{HeintzeKarcher78} which is an estimate for the volume of tubes around a sub-manifold which depends on the ambient sectional curvatures and the second fundamental form of the submanifold.  Here we, of course, must use ``weighted" notions of all of these quantities.  In the exposition below we will highlight how the arguments in \cite{HeintzeKarcher78} need to be modified in the weighted setting, and refer to the original text for background information. 

Let $H$ be a submanifold of Riemannian manifold $(M,g)$.  If $N$ is a  normal vector field to $H$, we will use the convention that the second fundamental form  of $H$ with respect to the metric $g$ and field $N$ is 
\begin{align*} 
\II^{g}_N(X,Y) = -g(N, \nabla_X Y) = g(\nabla_X N, Y)
\end{align*}
where $X,Y \in T_pH$.  Note that, traditionally $N$ is assumed to be a unit normal field, but it will aid our notation below to allow $N$ to be any normal field. The shape operator with respect to $g$ is then $S^{g}_N(X) = (\nabla_X N)^T$ so that $\II^{g}_N(X,Y)= g(S_N(X), Y)$. Clearly the shape operator only depends on the value of $N$ at the point.   Our estimates will not depend on the shape operator of $H$ with respect to $g$, but with respect to the conformal metric $\widetilde{g} = e^{-2\phi} g$. If $N$ is a  normal vector to $H$ with respect to $g$ then it is also a normal field in the metric $\widetilde{g}$.  Then there is a simple formula for  the second fundamental form and shape operator under conformal change:
\begin{align*}
\II^{\widetilde{g}}_{N}(X,Y) =  \II^{g}_{N}(X,Y) - d\phi(N)g(X,Y)\qquad S^{\widetilde{g}}_{N}(X) = S_N(X) - d\phi(N)X.
\end{align*}

The estimate for the volume of tubes we are after will depend on a Jacobi field comparison similar to the Rauch comparison theorem for $H$-Jacobi fields along focal point free geodesics.    Let $\sigma$ be a $g$-geodesic with $\sigma(0) \in H$ and $\sigma'(0) \in (T_pH)^{\perp}$. An $H$-Jacobi field  along $\sigma$ is a Jacobi field $J$ with $J(0) \in T_p H$ and $J'(0) - S^g_{\sigma'(0)} (J(0)) \in \left(T_pH\right)^{\perp}$. Equivalently, the $H$-Jacobi fields are the variation fields coming from variations of geodesics normal to $H$.  

For any vector field $V$ along $\sigma$ with $V(0) \in T_{\sigma(0)}H$, the $H$-index of $V$ is 
\begin{align*}
I^H(V,V)=  \II^g_{\sigma'}(V(0), V(0))  + \int_0^{t_0} |V'|^2 - g(R(V, \sigma')\sigma', V )dt
\end{align*}
A standard calculation shows that if $V(t_0) = 0$ then $\frac{d^2 E}{ds^2}\Big|_{s=0} = I^H(V,V)$.   $H$-Jacobi fields describe the derivative of the normal exponential map and are the minimizers of the $H$-index.   This implies the $H$-index lemma which states that if  $V$ is a vector field along $\sigma$ with $V(0) \in T_{\sigma(0)}H$ and $J$ is an $H$-Jacobi field such that $V(t_0) = J(t_0)$, then $I^H(J,J) \leq I^H(V,V)$ with equality if and only if $V=J$. See \cite[Chapter 10, section 4]{doCarmo76} for details.   From the proof of Proposition \ref{Prop:NewIndex}, we have the following formula for the $H$-index involving the weighted curvtaures. 
\begin{proposition}\label{Prop:NewHIndex}
Let $H$ be a submanifold of a manifold with density $(M,g,\phi)$.  Let $\gamma(s)$ be a  $\phi$-geodesic $\gamma:[a,b]\to M$ with standard re-parametrization which is normal to $H$ at $\gamma(a)$.  Then if  $V$ is a vector-field along $\gamma$ everywhere orthogonal to $\dot\gamma$, we have
\begin{align*}
I^H(e^\phi V,e^\phi V) & =  \II^{\phi}_{\dot{\gamma}(a)} (V,V) + d\phi(\dot\gamma(b))|V(b)|^2  \\
& \quad + \int_a^b \left(\left|\nabla_{\dot\gamma} V\right|^2 - g(R^{\nabla^\phi}(V,\dot\gamma)\dot\gamma,V)\right)ds.  
\end{align*}
where $\II^{\phi}_{\dot{\gamma}(a)} (X,Y)$ is defined to be $\II^{\widetilde{g}}_{\dot{\gamma}(a)} (X,Y).$ 
\end{proposition}

\begin{remark}
Given a $\phi$ geodesic as in the proposition, we will call $ \II^{\phi}_{\dot{\gamma}(a)}(X,Y)$ the \emph{weighted second fundamental form} with respect to  $\dot{\gamma}(a)$.  Similarly we call $S^{\phi}_{\dot{\gamma}}(X) = S^{\widetilde{g}}_{\dot{\gamma}}(X) $ the \emph{weighted shape operator} with respect to  $\dot{\gamma}(a)$.  The weighted second fundamental form and shape operator are a rescaling of the standard second fundamental form and shape operator of $\widetilde{g}$ with respect to a unit normal field. That is,  $\widetilde{N} = e^{-\phi(\gamma(a))} \dot{\gamma}$ is a unit vector in the $\widetilde{g}$ metric so that  
\[  \II^{\phi}_{\dot{\gamma}(a)} (X,Y)= e^{\phi(\gamma(a))} \II^{\widetilde{g}}_{\widetilde{N}} (X,Y) \qquad  S^{\phi}_{\dot{\gamma}(a)}(X) = e^{\phi(\gamma(a))}S^{\widetilde{g}}_{\widetilde{N}}(X)\]
\end{remark}
 
We now can state the weighted version of the Heintze-Karcher  comparison which measures the distortion of the volume form when pulled back via the normal exponential map by estimating the logarithmic derivative of a wedge product of $n-1$ linearly independent orthogonal Jacobi fields. 

\begin{lemma} \label{JacobiVolumeLemma}  Let $(M^n,g, \phi), (\widehat{M}^n, \widehat{g}, \widehat{\phi})$ be a Riemannian manifold with density  and let $H, \widehat{H}$ be a submanifolds of the same dimension of $M$ and $\widehat{M}$ respectively.  Let $\gamma, \widehat{\gamma}:[0,S] \rightarrow M, \widehat{M}$ be $\phi$-geodesics with standard reparametrization meeting $H, \widehat{H}$ perpendicularly at $s=0$ with no focal points on $[0,S]$.   
Let $Y_1, Y_2, \dots, Y_{n-1}$ be $n-1$ linearly independent $H$-Jacobi fields along $\gamma$ which are all perpendicular to $\dot{\gamma}$ and define  $\widehat{Y}_1, \widehat{Y}_2, \dots, \widehat{Y}_{n-1}$ similarly.  Suppose that $R^{\phi}\left(V, \frac{d\gamma}{ds}, \frac{d\gamma}{ds},  V\right) \geq \widehat{R} ^{\widehat{\phi}}\left(U,\frac{d \widehat{\gamma}}{ds}, \frac{d \widehat{\gamma}}{ds}, U\right)$ for all unit vectors $U$ and $V$ perpendicular $\gamma$ and $\widehat{\gamma}$ respectively. Suppose also that the eigenvalues of the weighted shape operators $\lambda_i$ and $\widehat{\lambda}_i$ satisfy $\lambda_i \leq \widehat{\lambda}_i$ for some ordering of the eigenvalues.  Then 
\begin{align*}
\frac{d}{ds}  \log \left( e^{-(n-1)\phi} \left| Y_1(s) \wedge \dots \wedge Y_{n-1}(s) \right| \right) \leq \frac{d}{ds}  \log \left( e^{-(n-1)\widehat{\phi}} \left| \widehat{Y}_1(s) \wedge \dots \wedge \widehat{Y}_{n-1}(s) \right| \right)
\end{align*}
In particular, 
\begin{align*}
 e^{-(n-1)\left(\phi(s)-\phi(0) \right)} \left| Y_1(s) \wedge \dots \wedge Y_{n-1}(s) \right| \leq  e^{-(n-1)\left(\widehat{\phi}(s)- \widehat{\phi}(0)\right)} \left| \widehat{Y}_1(s) \wedge \dots \wedge \widehat{Y}_{n-1}(s) \right|
\end{align*}
\end{lemma}

\begin{remark}  The inequality $\lambda_i \leq \widehat{\lambda_i}$ can also be  re-phased in terms of the second fundamental forms of $H$ and $\widehat{H}$ in the conformal metrics $e^{-2\phi}g$ and $e^{-2\widehat{\phi}}\widehat{g}$
with respect to unit normal vectors.  Call the eigenvalues of the conformal shape operators $\nu_i$ and $\widehat{\nu}_i$, then 
\[\lambda_i \leq \widehat{\lambda}_i \qquad  \Longleftrightarrow \qquad e^{\phi(\gamma(a))}\nu_i \leq e^{\widehat{\phi}(\widehat{\gamma}(a))}\widehat{\nu}_i.\]
In particular, we see that the inequality holds if $H$ and $\widehat{H}$ are both totally geodesic submanifolds with respect to the conformal metrics. 
\end{remark}

\begin{proof}
Fix $s=s_1$, by taking linear combinations of the $Y_i$ we can assume that $Y_i, \widehat{Y}_i$ are orthonormal at $s_1$ without changing the logarithmic derivative.  Then, at $s_1$ we obtain 
\begin{align}
\frac{d}{ds}\Big|_{s=s_1}&\log \left( e^{-(n-1) \phi} \left| Y_1 \wedge Y_2 \wedge \dots \wedge Y_{n-1} \right| \right) \nonumber\\
&= -(n-1)  d \phi \left( \frac{d \gamma}{ds} \right)(s_1) + \sum_{i=1} ^{n-1} g\left ( \frac{d}{ds} Y_i, Y_i\right)(s_1) \nonumber\\
&= \sum_{i=1}^{n-1} \left(   -d \phi \left( \frac{d \gamma}{ds}(s_1) \right) + e^{2\phi(s_1)} I^H_{s_1} (Y_i, Y_i) \right) \label{eqn:LogDerivative}
\end{align}
and similarly for the $\widehat{Y_i}$.  We can further assume by taking linear combinations  that either $\widehat{Y}_i(0) = 0$ or $\widehat{Y}_i(0)$ is an eigenvector for $S_{\widehat{\gamma}'(0)}^{\phi}$. Following \cite[3.4.7]{HeintzeKarcher78}, define 
\begin{align*}
W_i(s) = e^{-\phi(s_1)+ \widehat{\phi}(s_1)} P_{s} \circ \iota \circ \widehat{P}_{-s}\left(e^{-\widehat{\phi}(s)} \widehat{Y}_i(s)\right)
\end{align*}
 where  $\iota$ is a linear isometry from $T_{\gamma(0)} M$ to $T_{\widehat{\gamma}(0)} \widehat{M}$ which takes $T_{\gamma(0)} H$ to $T_{\widehat{\gamma}(0)} \widehat{H}$ and such that $\iota\left(\dot{\gamma}(0) \right)$  is parallel to  $\dot{\widehat{\gamma}}(0)$,  $P_{s}$ is ($\nabla$)-parallel translation along $\gamma$, and $\widehat{P}$ is $\widehat{\nabla}$-parallel translation along $\widehat{\gamma}$.   

Then $W_i$ is a variation field along $\sigma$ with $W_i(0)\in T_{\sigma(0)} H$.  We also have that $\{W_i(s_1)\}_{i=1}^{n-1}$ is an orthogonal basis of the normal space to $\dot{\gamma}(s_1)$ with  $|W_i(s_1)| = e^{-\phi(s_1)}$, so by changing the $Y_i$ again via constant coefficients, we can assume that $Y_i(s_1) = e^{\phi(s_1)} W_i(s_1)$.

By the $H$-index Lemma and Proposition \ref{Prop:NewIndex} we have
\begin{align*}
I^H_{s_1}(Y_i, Y_i)& \leq I^H_{s_1}(e^{\phi(s)}W_i, e^{\phi(s)}W_i) \\
&= \II^{\phi}_{\gamma'(0)}(W_i(0), W_i(0)) + d\phi(\dot\gamma)(s_1) |W_i|^2(s_1) \\
& \qquad  + \int_0^{s_1}\left(\left|\nabla_{\dot\gamma} W_i\right|^2 - g(R^\phi(W_i,\dot\gamma)\dot\gamma,W_i)\right)ds.
\end{align*}
Combing this with (\ref{eqn:LogDerivative}) gives
\begin{align*}
\frac{d}{ds}\Big|_{s=s_1}&\log \left( e^{-(n-1) \phi} \left| Y_1 \wedge Y_2 \wedge \dots \wedge Y_{n-1} \right| \right)\\
&\leq e^{2\phi(s_1)}   \sum_{i=1}^{n-1} \left(  \II^{\phi}_{\gamma'(0)}(W_i(0), W_i(0))  + \int_0^{s_1}\left(\left|\nabla_{\dot\gamma} W_i\right|^2 - g(R^\phi(W_i,\dot\gamma)\dot\gamma,W_i)\right)ds\right)\\
&= e^{2\phi(s_1)}   \sum_{i=1}^{n-1} \left( \lambda_i |W_i(0)|^2 +  \int_0^{s_1}\left(\left|\nabla_{\dot\gamma} W_i\right|^2 - |W_i|^2g\left(R^\phi\left(\frac{W_i}{|W_i|},\dot\gamma\right)\dot\gamma,\frac{W_i}{|W_i|}\right)\right)ds\right)
\end{align*}
On the other hand, applying  Proposition \ref{Prop:NewIndex} to $V=e^{-\widehat{\phi}(s)} \widehat{Y_i}(s)$ gives
\begin{align*}
I^{\widehat{H}}_{s_1}(\widehat{Y_i}, \widehat{Y_i})
&= e^{-2\widehat{\phi}(0)} \II^{\widehat{\phi}}_{\widehat{\sigma}'(0)}(\widehat{Y_i}(0), \widehat{Y_i}(0))  +e^{-2\widehat{\phi}(s_1)} d\widehat{\phi}\left(\dot{\widehat{\gamma}}(s_1)\right)\\
& \quad +  \int_0^{s_1} \left(\left|\nabla_{\dot{\widehat{\gamma}}} e^{-\widehat{\phi}} \widehat{Y}_i\right|^2 - e^{-2\widehat{\phi}(s)} \widehat{g}(\widehat{R}^{\widehat{\phi}}(\widehat{Y}_i,\dot{\widehat{\gamma}})\dot{\widehat{\gamma}},\widehat{Y_i})\right)ds 
\end{align*}
So we have
\begin{align*}
\frac{d}{ds}\Big|_{s=s_1}& \log \left( e^{-(n-1)\widehat{\phi}} \left| \widehat{Y}_1(s) \wedge \dots \wedge \widehat{Y}_{n-1}(s) \right| \right)\\
 &= \sum_{i=1}^{n-1} e^{2\left(\widehat{\phi}(s_1) -\widehat{\phi}(0)\right)} \II^{\widehat{\phi}}_{\widehat{\sigma}'(0)}(\widehat{Y_i}(0), \widehat{Y_i}(0)) \\
 & \quad   + e^{2\widehat{\phi}(s_1)}  \int_0^{s_1} \left(\left|\nabla_{\dot{\widehat{\gamma}}} e^{-\widehat{\phi}} \widehat{Y}_i\right|^2 - e^{-2\widehat{\phi}(s)} \widehat{g}(\widehat{R}^{\widehat{\phi}}(\widehat{Y}_i,\dot{\widehat{\gamma}})\dot{\widehat{\gamma}},\widehat{Y_i})\right)ds  \\
 &= \sum_{i=1}^{n-1} e^{2\left(\widehat{\phi}(s_1) -\widehat{\phi}(0)\right)}\widehat{\lambda}_i |\widehat{Y}_i(0)|^2 \\
 & \quad   + e^{2\widehat{\phi}(s_1)}  \int_0^{s_1} \left(\left|\nabla_{\dot{\widehat{\gamma}}} e^{-\widehat{\phi}} \widehat{Y}_i\right|^2 - e^{-2\widehat{\phi}(s)} |\widehat{Y}_i|^2 \widehat{g}\left(\widehat{R}^{\widehat{\phi}}\left(\frac{\widehat{Y}_i}{|\widehat{Y}_i|},\dot{\widehat{\gamma}}\right)\dot{\widehat{\gamma}},\frac{\widehat{Y_i}}{|\widehat{Y}_i|}\right)\right)ds \\
 &\geq  e^{2\phi(s_1)}   \sum_{i=1}^{n-1} \left( \lambda_i |W_i(0)|^2 +  \int_0^{s_1}\left(\left|\nabla_{\dot\gamma} W_i\right|^2 - |W_i|^2g\left(R^\phi\left(\frac{W_i}{|W_i|},\dot\gamma\right)\dot\gamma,\frac{W_i}{|W_i|}\right)\right)ds\right)
 \end{align*}
 Where, in the last line,  we have used the hypotheses of the theorem along with the fact that 
  \begin{align*}
 e^{\widehat{\phi}(s_1) - \widehat{\phi}(s)} |\widehat{Y}_i(s)| & = e^{\phi(s_1)}|W_i(s)| \\
 e^{\widehat{\phi}(s_1)} | \nabla_{\dot{\widehat{\gamma}}} e^{-\widehat{\phi}(s)} \widehat{Y}_i(s) | &= e^{\phi(s_1)} |\nabla_{\dot{\gamma}} W_i(s)|
 \end{align*}
 which comes from the definition of $W_i$. This proves the lemma.
 \end{proof}

 Now we consider volume comparison.  There are two special cases where only a Ricci curvature assumption is needed to control the tube volume: when $H$ is a point and  when $H$ is a hypersurface.  These cases have already appeared in the literature, see \cite[Theorem 4.5]{WylieYeroshkin} and  \cite[Theorem 1.3]{LiXia},   \cite[Theorem 1.4]{Sakurai} respectively.  
 
 Otherwise, let $H$ be an isometrically immersed submanifold of $M$ with normal bundle $\pi: \nu(H) \rightarrow H$.   Let $\exp^{\perp}: \nu(H) \rightarrow M$ be the normal exponential map of $H$.  For a full exposition on how the wedge of Jacobi fields controls the volume distortion of the normal exponential map, see \cite[Sections 2 \& 3]{HeintzeKarcher78}.  The comparison space for our volume comparison will be the same as is used in \cite{HeintzeKarcher78}, a tube with constant radial curvatures around $H$ (with no density).   
  
We also require a weighted version of the mean curvature vector to state our most general results. Recall that the vector-valued second fundamental form is the unique map $T_pH \times T_pH \rightarrow (T_pH)^\perp$ such that $g(h(X,Y), N) = \II_N(X,Y)$ for all $N \in T_pH^{\perp}$.  We define the weighted version $h^{\phi}$ via the same formula with respect to $\II^{\phi}$.  Then we obtain 
\begin{align*}
g(h^{\phi}(X,Y), N) &= \II^{\phi}_N(X,Y) \\
&= \II_N(X,Y) - d\phi(N)g(X,Y) \\
&= g(h(X,Y) - g(X,Y)\nabla{\phi}, N)
\end{align*}
so that $h^{\phi}(X,Y) = h(X,Y) - g(X,Y) (\nabla \phi)^{\perp}$ where $\perp$ denotes the orthogonal projection from $T_pM$ to $(T_pH)^{\perp}$. Define the weighted mean curvature vector as  $\eta^{\phi} = \frac{ \tr(h^{\phi})}{\dim(H)} = \eta- (\nabla \phi)^{\perp}$. Where $\eta$ is the usual mean curvature vector. Following the notation of \cite{HeintzeKarcher78}, also let $A^{\phi} = |\eta|$ and $\Lambda^{\phi}(H) = \sup_H A^{\phi}$.

The appearance of the reparametrized distance parameter $s$ in Lemma \ref{JacobiVolumeLemma} also adds some technical considerations.   We will have two different versions of the volume comparison.  The first will be for distance tubes that is  $T(H, r) = \{ x : d(x, H) \leq r\}$ which we call that tube around $H$ with radius $r$. This comparison will be in terms of the $f$-volume $e^{-f} dvol_g$ where $f =(n-1) \phi$.  
The second volume comparison will be for the re-parametrized tubes around $H$.  We then define the re-parametrized tube as 
\begin{align*}
\widetilde{T}(H,s) = \{ x: \exists y \in H, s(x, y) \leq s \} 
\end{align*}
For the re-parametrized tubes we also use re-parametrized volume $\mu(A) = \int_A e^{-(n+1) \phi} dvol_g$.  
 
We  define the comparison function $J^{\phi}_{\kappa}$ as  
\begin{align*}
J^{\phi}_{\kappa}(p,r,  \theta) &= \left(cs_{\kappa}(s(p, r, \theta)) - g(\eta^{\phi}(p), \theta) sn_{\kappa}(s(p,r,  \theta)) \right)^m sn_{\kappa}(s(p, r, \theta))^{n-m-1}
\end{align*}
Where $\mathrm{dim}(H)=m$, $p \in H$, $s(p, r, \theta)$ is the re-parametrized distance between the point $p$ and the point of distance $r$ from $p$ along a geodesic with initial velocity $\theta$, and $\eta^{\phi}(p)$ is the weighted mean curvature normal vector to $H$ at $p$.  We also define 
\begin{align*}
J_{\kappa}(p, s) &=\left(cs_{\kappa}(s) - g(\eta^{\phi}(p), \theta) sn_{\kappa}(s) \right)^m sn_{\kappa}(s)^{n-m-1}
\end{align*}
Define $z_{\kappa} (p,\theta)$ to be the smallest positive number $r_0$ such that $J_{\kappa}(p, r_0, \theta) = 0$ and $\widetilde{z}_{\kappa}(p, \theta)$ be the value of $s$ defined similarly for $J_{\kappa}(p, s)$. 
Our volume comparison theorem is the following. 

\begin{theorem} \label{Thm:TubeVol}
Suppose that $H$ is an $m$-dimensional isometrically immersed in a manifold $(M^n,g,\phi)$ with $\overline{sec}_{\phi} \geq \kappa e^{-4\phi}$ then 
\begin{enumerate}
\item 
\[ \mathrm{vol}_f(T(H, r)) \leq \int_N \left( \int_{S^{n-m-1}} d\theta \int_{0}^{\min\{r, z(p, \theta)\}}  J^{\phi}_{\kappa}(p,r,  \theta) dr\right) e^{-f(p)} dvol_N\] 
\item 
\[ \mu(\widetilde{T}(H, s)) \leq \int_N \left( \int_{S^{n-m-1}} d\theta \int_{0}^{\min\{s, \widetilde{z}(p,\theta)\}} J_{\kappa}(p,s) ds \right)e^{-f(p)} dvol_N\]
\end{enumerate}
\end{theorem}
 
 \begin{proof}
 Given a unit normal vector $\theta$ to $H$, let $\mathrm{foc}(\theta)$ be the supremum of the values of  $r$ such that the unique geodesic with initial velocity $\theta$ has no focal point to $H$ to distance $r$. We then have that 
 \begin{align*}
 \mathrm{vol}_f (T(H, r))= \int_N \left( \int_{S^{n-m-1}} d\theta \int_{0}^{\min\{foc(\theta), r\}} e^{-f}|\det(d\exp^{\perp}_{\theta})| dr\right) dvol_N
 \end{align*}
 On the other hand, we can estimate $|\det(d\exp^{\perp}_{\theta})|$ as 
 \begin{align*}
 |\det(d\exp^{\perp}_{\theta})| = \frac{|(d\exp^{\perp}_{\theta})(u_1) \wedge \dots \wedge (d\exp^{\perp}_{\theta})(u_n)|}{|u_1 \wedge \dots \wedge u_n|}
 \end{align*}
 Where $u_i$ is any basis of $T_{\theta} \nu(H)$. A natural choice for $u_i$ is a basis of $H$-Jacobi fields along the geodesic, which is achieved by taking $u_i$ to be suitable linear vector fields along the geodesic.  Then one has $(d\exp^{\perp}_{\theta})(u_i) = Y_i$ a normal Jacobi field.  Let $(\widehat{M}, \widehat{g})$ be the ``canonical"  metric  on $\nu(H)$ as described in \cite[3.1.1]{HeintzeKarcher78}, with a constant density.  Then  the function $J_{\kappa}$  is exactly 
 \begin{align*}
 |\det(d\exp^{\perp}_{\theta})| = \frac{|\widetilde{Y_i} \wedge \dots \wedge \widetilde{Y}_{n-1}|}{|u_1 \wedge \dots \wedge u_n|}.
 \end{align*} in $\widehat{M}$.    Since the $u_i$ are independent of the manifold chosen, Lemma  \ref{JacobiVolumeLemma} then gives a comparison between the volume forms in the corresponding spaces.
 \begin{align*}
 e^{-f}|\det(d\exp^{\perp}_{\theta})|  \leq e^{-f(p)}J^{\phi}_{\kappa}(p,r,  \theta) 
 \end{align*}
 This give the first part of the theorem.  For the second part of the theorem define $\mathrm{foc}_{s} (\theta)$ be the value of the integral $\int_0^{\mathrm{foc}(\theta)} e^{-2\phi(\gamma(t))} dt$  where $\gamma$ is the geodesic with $\gamma(0)=p$ and $\gamma'(0) = \theta$. Then we can write 
 \begin{align*}
 \mu(\widetilde{T}(H, s)) &= \int_N \left( \int_{S^{n-m-1}} d\theta \int_{0}^{\min\{foc_s(\theta), s(p, r, \theta)\}} e^{\frac{-(n+1)}{n-1} f}|\det(d\exp^{\perp}_{\theta})| dr\right) dvol_N
 \end{align*}
 Making the change of variable $ds = e^{\frac{-2f}{n-1}} dr$, along with using the volume element comparison as above, gives us 
  \begin{align*}
 \mu(\widetilde{T}(H, s)) &\leq\int_N \left( \int_{S^{n-m-1}} d\theta \int_{0}^{s} J_{\kappa}(p,s)ds\right) e^{- f(p)} dvol_N
 \end{align*}
 \end{proof}
 
We note that the advantage of the comparison (1) is that  it is in terms of the distance tubes, however the comparison integral on the right hand side is impossible to compute with out more information about $f$ as the functions $s(p,r, \theta)$ depend on $f$.  This comparison is useful, however if we assume some bounds on the function $f$.  On the other hand, in comparison (2) it is hard to compute the sets $\widetilde{T}(H, s)$,but the comparison function on the right hand side is computable and exactly the tube volume of the corresponding un-weighted model space.   Moreover, we  note that by Theorem 2.2 of \cite{WylieYeroshkin}, for example, if $\kappa>0$ then $\mathrm{sup}_{p,q \in M} s(p,q) \leq \frac{\pi}{\sqrt{\kappa}}$, so in this case we can use (2) to get a uniform upper bound on $\mu(M)$ in terms of the data on $H$. 

Using either (1) or (2) we obtain the following result when we assume $\phi$ is bounded.  
 \begin{corollary} \label{Cor:VolCompSubMan}
Suppose that $(M^n,g,\phi)$ is a compact  Riemannian manifold with $\overline{\sec}_{\phi} \geq \kappa e^{-4\phi}$, $|\phi|\leq B$, and $\mathrm{diam}(M) \leq D.$  Then for any submanifold $H^m$ of $M$ there is an explicit positive constant $C(n, m, \kappa, B, D, \Lambda^{\phi}(H))$ such that 
\[\mathrm{vol}(M) \leq C \vol(H). \]
\end{corollary} 

By applying the theorem to the conformal metric $\widetilde{g}$ we obtain the following result for closed geodesics. 

\begin{corollary} \label{Cor:WeightedCheeger} Let $(M^n, g, \phi)$ be a compact manifold with density such that $\overline{\sec}_{\phi} \geq \kappa e^{-4\phi}$, $|\phi|\leq B$, $\mathrm{diam}(M) \leq D$ and $\mathrm{vol}(M) \geq v$ then there is a constant $L=L(n,\kappa, B,D,v)$ such that any closed geodesic $\sigma$ in $M$ has length greater than or equal to $L$. 
\end{corollary}

\begin{proof}
Let $\sigma$ be a closed geodesic in $(M,g)$.  Then $\sigma$ has vanishing weighted second fundamental form in the  manifold with density $(M, \widetilde{g}, -\phi)$.  Computing the weighted curvature of $(\widetilde{g}, -\phi)$ we have the relation $\overline{\sec}_{ \widetilde{g}, -\phi}(X,Y) = e^{2\phi} \overline{\sec}_{g, \phi}(Y,X)$ (see Proposition 2.1 of \cite{Wylie15}).  Since $\phi$ is uniformly bounded this gives a uniform constant $\widetilde{k}$  such  that $\overline{\sec}_{ \widetilde{g}, -\phi} \geq \widetilde{k} e^{4\phi}$. We can also trivially estimate $\mathrm{diam}(M, \widetilde{g})$, and $\mathrm{vol}_{\widetilde g}(M)$ uniformly in terms of  $n, B, D$ and $v$.  Applying Corollary \ref{Cor:VolCompSubMan} gives a lower bound on the $\widetilde{g}$-length of $\sigma$. Since $|\phi|\leq B$ this also gives the desired bound on the $g$-length. 
\end{proof}

This result combined with the results above allows us to establish the most general finiteness theorem, Theorem \ref{IntroThm:GenFinite}. 

\begin{reptheorem}{IntroThm:GenFinite}
For given $n\geq 2$, $a,v, D, k >0$ the class of compact Riemannian manifolds with 
\[
\mathrm{diam(M)} \leq D, \quad \mathrm{vol}(M) \geq v, \quad  \overline{K}(a) \leq k, \quad \text{and} \quad  \underline{\kappa}(a) \geq -k
\]
contains only finitely many diffeomorphism types. 
\end{reptheorem}

\begin{proof}
Let $\phi$ be a function so that $\overline{\sec}_{\phi} \geq -2k e^{-4\phi}$ with $|d\phi|\leq a$.  Since the diameter is bounded and, by Remark \ref{Rem:Zero}, we can choose $\phi$ so that there is a point where $\phi(p) = 0$, there is a constant $B$ depending on $D$ and $a$ such that $|\phi| \leq B$.    Theorem \ref{Thm:HessianComp} then provides the required two sided bounds on the Hessian of the distance function, so we only require a lower bound on injectivity radius to prove $C^{\alpha}$ compactness. A classical result of Klingenberg states that the injectivity radius is the smaller of the conjugate radius and the length of the smallest closed geodesic.    Lemma \ref{Lemma:ConjPos} gives the  lower bound on the conjugate radius and Corollary \ref{Cor:WeightedCheeger} gives the lower bound on the length of closed geodesics. 

\end{proof}

\begin{remark}
We have only  delved into results that follow from standard convergence theory as an illustration of the application of our comparison results.  It seems likely that one can improve these finiteness theorems by more fully developing the convergence theory in the weighted setting.  One should also be able to relax the point-wise sectional curvature bounds to integral bounds on the curvature tensor, as is done in the un-weighted setting.  It also seems likely that the pointwise derivative bound we impose on the density can be relaxed to integral bounds.  A more interesting question is whether $C^0$ bounds on the potential function suffices for finiteness theorems. We have only used the $C^1$-bounds in controlling the Hessian, indicating that there may be a wiser choice of coordinates than the standard ``distance coordinates" that perhaps take the density function $\phi$ into account.  
\end{remark}

\subsection{Radial Curvature Equation}

In the exposition above we have chosen to present the comparison theory for weighted sectional curvatures in terms of  Jacobi field estimates. However, just as in the un-weighted setting, these results can also be interpreted in  terms of the variation of shape operators of hypersurfaces.  Though we don't take this approach in any our applications, we show how it can easily be done once we have the definition of weighted second fundamental form as described in the previous section. 

 Given a submanifold $H$ and a normal vector field $N$, recall that the  modified shape operator  is  
\[S^{\phi}_N(X)= \nabla^{\phi} _X N =  \nabla_X N - d\phi(X)N - d\phi(N) X. \]
The following equation shows how to compute the curvatures normal to a hypersurface from the modified shape operator. 
\begin{proposition}[Radial Curvature Equation]
  Let $H$ be a hypersurface and $N$ a normal vector to $H$, then 
  \begin{align*}
  \left(\nabla_N^{\phi} S_N^{\phi}\right)(X) + \left(S_N^{\phi}\circ S_N^{\phi}\right)(X) = \nabla_X^{\phi}\left(S_N^{\phi}(N)\right) - R^{\nabla^\phi}(X,N)N 
  \end{align*}
  \end{proposition}
  \begin{proof}
  Consider 
  \begin{align*}
  R^{\phi}(X,N)N  &= \nabla^{\phi}_X \nabla^{\phi}_N N - \nabla^{\phi}_{[X, N]} N \\
  &= \nabla^{\phi}_X\left(S_N^{\phi}(N) \right) - \nabla^{\phi}_N \left(S^{\phi}_N(X)\right) + \nabla^{\phi}_{\nabla^{\phi}_N X} N - \nabla^{\phi}_{\nabla^{\phi}_X N} N \\
  &=  - \left(\nabla^{\phi}_N S_N^{\phi}\right)(X) -\left(S_N^{\phi}\circ S_N^{\phi}\right)(X) +  \nabla^{\phi}_X\left(S_N^{\phi}(N) \right)
    \end{align*}
  \end{proof}
To see the connection to Jacobi fields and the Hessian of the distance function  we apply the Radial Curvature Equation to the case where $H$ is a distance tube.  Let $A$ be a closed subset of $M$, and let $d_A(\cdot)$ be the Riemannian distance to $A$.  In a neighborhood of a point where $d_A$ is smooth we can let $\frac{d}{ds}= e^{2\phi} \nabla r$.  Then $\frac{d}{ds}$ is a normal vector for the distance tubes of $A$, i.e $T_r (A) = \{ x : d_A(x) = r\}$ and is a geodesic field for $\nabla^{\phi}$.  Letting $N = \frac{d}{ds}$ and $S = S_N$ we obtain 
  \begin{align}
   \left(\nabla_{\frac{d}{ds}}^{\phi} S^{\phi}\right)(X) + \left(S^{\phi}\circ S^{\phi}\right)(X) &= - R^{\nabla^\phi}\left(X,\frac{d}{ds}\right)\frac{d}{ds} \label{eqnRadialCurv}
   \end{align}
   since $S^{\phi}\left( \frac{d}{ds} \right) = \nabla^{\phi}_{\frac{d}{ds}} \frac{d}{ds}=0$. 
   
In this case 
\begin{align*}
S^{\phi}(X)& = \nabla^{\phi}_X  \left(e^{2\phi} \nabla r \right) \\
&=\nabla_X  \left(e^{2\phi} \nabla r \right) - e^{2\phi} d\phi(X) \frac{\partial}{\partial r}  -  e^{2\phi} d\phi(\nabla r)X\\
&= e^{2\phi} \left( \nabla_X \nabla r + d\phi(X)\nabla r - d\phi(\nabla r)X\right)
\end{align*}

These equations give us the following estimate for the derivative of the  weighted second fundamental form with respect to $\frac{d}{ds}$. 

\begin{proposition} \label{Prop:EvolveSecondFunForm}
Let $\gamma$ be a standard reparametrization of a minimizing geodesic  and let $\II^{\phi}$ be the second fundamental form of the distance tube to $\gamma(0)$ so that $\frac{d}{ds} =  \dot{\gamma}$.  Let $X$ and $Y$ are parallel fields along $\gamma$ which are also perpendicular to $\gamma$, i.e  $\nabla_{\dot{\gamma}} X = \nabla_{\dot{\gamma}} Y = 0$, $g\left(\frac{d}{ds}, X\right) =  \left(\frac{d}{ds}, Y\right) = 0$ then 
\begin{align}
\frac{d}{ds} \left( \II^{\phi}(X,Y) \right) &=  -g\left(\left(S^{\phi}\circ S^{\phi}\right)(X), Y\right)-g\left( R^{\phi}\left(X,\frac{d}{ds}\right)\frac{d}{ds}, Y \right)  \label{EvolveSecondFunForm}
\end{align}
\end{proposition}

\begin{proof}

We have
\begin{align*}
g(S^{\phi}(X), Y) &=e^{2\phi} \left( \mathrm{Hess} r(X,Y) + d\phi(X) dr(Y) - d\phi(\nabla r)g(X,Y) \right)
\end{align*} 
So, for $X,Y \perp \dot{\gamma}$,   $g(S^{\phi}(X), Y)= \II^{\phi}(X,Y)$.  Then we obtain 
\begin{align*}
\frac{d}{ds} \left(g(S^{\phi}(X), Y)\right) &= g\left( \nabla_{\frac{d}{ds}} \left(S^{\phi}(X)\right), Y\right) \\
&= g\left( \nabla^{\phi}_{\frac{d}{ds}} \left(S^{\phi}(X)\right), Y\right) + d\phi\left( \frac{d}{ds} \right) g(S^{\phi}(X), Y) \\
&=g\left(  \left(\nabla^{\phi}_{\frac{d}{ds}}S^{\phi}\right)(X), Y\right) + g\left( S^{\phi}\left(\nabla^{\phi}_{\frac{d}{ds}} X \right), Y \right) + d\phi\left( \frac{d}{ds} \right) g(S^{\phi}(X), Y)\\
&=g\left(  \left(\nabla^{\phi}_{\frac{d}{ds}}S^{\phi}\right)(X), Y\right) -d\phi\left( \frac{d}{ds} \right) g(S^{\phi}(X), Y) - d\phi(X) g\left( S^{\phi}\left(\frac{d}{ds} \right), Y \right) \\
& \qquad + d\phi\left( \frac{d}{ds} \right) g(S^{\phi}(X), Y)\\
&= g\left(  \left(\nabla^{\phi}_{\frac{d}{ds}}S^{\phi}\right)(X), Y\right)
\end{align*}

Then from the radial curvature equation we have 
\begin{align*}
\frac{d}{ds}g(S^{\phi}(X), Y)  &=  -g\left(\left(S^{\phi}\circ S^{\phi}\right)(X), Y\right)-g\left( R^{\phi}\left(X,\frac{d}{ds}\right)\frac{d}{ds}, Y \right) 
\end{align*}

\end{proof}

\begin{remark}
Tracing (\ref{EvolveSecondFunForm}) over the orthogonal complement of the geodesic gives Lemma 4.1 on \cite{WylieYeroshkin}. It is not hard to see that this equation could also be used to derive Theorem \ref{Thm:HessianComp}. 
\end{remark}


\end{document}